\documentclass[11pt, oneside]{article}   	
\usepackage{geometry}                		
\usepackage{graphicx}					
\usepackage{amssymb}
\usepackage{amsmath}
\usepackage{amsthm}
\usepackage{mathabx}
\usepackage{dsfont}
\usepackage{mathtools}
\usepackage{tikz}
\usepackage{todonotes}
\usepackage{hyperref}
\usepackage{constants}
\usepackage{enumerate}
\usepackage{enumitem}
\providecommand*{\turnangle}{\rotatebox[origin=c]{180}{$\angle$}}

\usepackage{cleveref}
\newconstantfamily{c}{symbol=c}

\newtheorem{The}{Theorem}[section]
\newtheorem{lemma}[The]{Lemma}
\newtheorem{proposition}[The]{Proposition}
\newtheorem{Cor}[The]{Corollary}
\theoremstyle{definition}

\theoremstyle{remark}

\newtheorem{Rk}[The]{Remark}
\newtheorem{Rks}[The]{Remarks}

\newcommand{\Pannealed}{\mathbb{P}}	
\newcommand{\Eannealed}{\mathbb{E}}	

\newcommand{\Penv}{\mathbf{P}}	

\newcommand{\Qcoupling}{\mathbb{Q}}	
\newcommand{\Ecoupling}{\mathbb{E}_{\mathbb{Q}}}

\newcommand{\pfull}{p_\bullet}
\newcommand{\pempty}{p_{\circ}}

\newcommand{\integers}{\mathbb{Z}}	

\newcommand{\cI}{\mathcal{I}}
\newcommand{\1}{\mathds{1}}

\newcommand{\dE}{\mathbb{E}}
\newcommand{\dN}{\mathbb{N}}
\newcommand{\dP}{\mathbb{P}}
\newcommand{\dQ}{\mathbb{Q}}
\newcommand{\dR}{\mathbb{R}}
\newcommand{\dZ}{\mathbb{Z}}

\newcommand{\nonlazy}{\alpha}
\newcommand{\lazy}{1-\alpha}

\newcommand{\posreals}{\mathbb{R}_+}

\newcommand{\Poisson}{\text{Poi}}

\newcommand{\gui}{\color{blue}}

\newcommand{\tmpdel}{\color{green}}
\newcommand{\tmpend}{\color{black}}

\title{Law of Large Numbers for a random walk on dynamic environments with drift}
\author{Guillaume Conchon-\,\hspace{-.3mm}-Kerjan\thanks{Department of Mathematics, King's College London, WC2R 2LS, United Kingdom \newline Email: \texttt{guillaume.conchon-kerjan/toril.palaniappan@kcl.ac.uk}.}~\thanks{GCK is grateful to EPSRC for support through the grant UKRI2781.} \and Toril Palaniappan\footnotemark[1]}

\begin{document}
\maketitle

\begin{abstract}
We study a random walk driven by a particle system from a generic class, and establish a law of large numbers for the walk for almost all densities of the environment.
\\
To do so, we exploit the finite-ranged approximations of the environment from \cite{CKKR24} in a new way, whereby the monotonicity (in the density) of the walker's displacement is leveraged to show the existence of an actual speed. This bypasses the constructions in \cite{HKT20} and generalises its Theorem 1.1, which applied to specific environments. 
\\
We illustrate this with a family of particle systems where the particles have underlying drifts, namely mixtures of APCRWs (Asymmetric Poisson Cloud of Random Walks). In particular, when all particles have the same drift, we prove the LLN under any choice of parameters save one critical density of the environment. To our knowledge, this is the first time that such a conservative and slow-mixing environment with drift is treated outside of the non-nestling case (in which the walker is already assumed to travel strictly faster/slower than the drift~\cite{BallisticRWsLateralDecoupling}).
\end{abstract}


\section{Introduction}



{
Random walks in random environments is a prominent topic in probability theory. Over the past few decades, a wide variety of models have emerged, one motivation being to capture phenomena beyond the rigidity of a constant homogeneous environment, potentially escaping the universality of the Brownian motion. A notable early instance are i.i.d. static environments~\cite{KKS}, \cite{Solomon}, which display anomalous (subdiffusive) fluctuations, and transience with zero-speed in dimension one - some analogous questions in higher dimensions being still open after long-standing efforts (see e.g. \cite{zbMATH06028501, SznitmanConj,zbMATH02070282} and associated references). 
\\
Coming back to what Brownian motion was originally meant to model - the motion of pollen grains on water - a natural idea is to imagine a random walk advected by a particle system~\cite{Avenaal1HydroLim,BlondelHilarioTeixeira, HHSST15,HS15,JaraMenezes,MountfordVares}. But even the simplest instances, from a (discrete) probabilistic point of view, pose significant technical challenges. First, classical techniques tailored for static environment (e.g. from potential theory) become inefficient when the environment is dynamic. Second, standard models for particle systems such as exclusion processes, or even clouds of independent random walks, are conservative and exhibit slow-mixing properties, yielding strong space-time correlations and escaping the frameworks such as in~\cite{Allasia} or~\cite{BlondelHilarioTeixeira}. Third, the random walk on such systems is generally non-reversible (though the particle system itself may be - see e.g.~\cite{CKKR24} for details), so that many elaborate schemes to derive an invariant measure from the point of view of the walk (as for instance in~\cite{zbMATH03944984,zbMATH03951720,zbMATH06458602}) do not apply. Notable recent progress on environments of driftless particles include laws of large numbers~\cite{HKT20} and a strict monotonicity principle~\cite{CKKR24} (in dimension 1 - see e.g.~\cite{bethuelsen2024randomwalksrandomwalks} for dimension $\geq 5$).  
\\
When the environment particles have an inherent drift, further technical difficulties may appear, in particular, the use of lateral decoupling (see e.g.~\cite{HKT20} and references therein) becomes more delicate. This difference is striking in the case of the exclusion process, a system of non-overlapping particles which in the zero-drift case (SEP, simple symmetric exclusion process) has convenient structural properties (e.g. reversible and equivalent to an interchange process) that its asymmetric version, the ASEP, does not possess. While there is now abundant literature on random walks advected by the SEP (see e.g.~\cite{AvenaBlondelFaggionato,dosSantosBoundsSpeed} and aforementioned works), results on environments with a drift are much more sparse to the best of our knowledge, and require a strong preliminary ballisticity assumption \cite{BallisticRWsLateralDecoupling}.
We discuss this in greater detail in Section~\ref{subsec:discussion}.
\\
In this work, we derive directly the law of large numbers for the wide class of environments introduced in~\cite{CKKR24} (which includes the SEP and PCRW among others), and apply it to the Asymmetric Poisson Cloud of Random Walks (APCRW), which can be seen as a more amenable variant of the ASEP.

\subsection{Model} \label{sec: model}
\textbf{The environment.} \\ 
The APCRW consists of a collection of independent lazy drifted random walks. It depends on three parameters: the density $\rho > 0$, and $\alpha,q \in (0,1)$.
\\
Formally, it is a discrete-time process  $\eta = \left(\eta_t(x) : x \in \integers, t \geq 0 \right)$ such that at time 0, each site $x\in \integers$ is independently given $\eta_0(x)\sim\Poisson(\rho)$ particles. All of these particles perform independent discrete-time random walks with laziness $1-\alpha$ and drift $2q-1$. That is, they have probability $1-\alpha$ to stay put, and $\alpha q$ (resp. $\alpha(1-q)$)  to jump to their right (resp. left) neighbour. $\eta_t(x)$ then denotes the number of particles at position $x$ at time $t$. 


Some of our results will work for a general environments class, which satisfy a specific set of conditions detailed in Sections~\ref{subsec:DefRWRDREgeneral} and~\ref{subsec:conditions}, and to which the APCRW belongs. Roughly speaking, we assume at \ref{pe:markov} - \ref{pe:density} that the environments is Markov in time, has a stationary measure (namely $\Poisson(\rho)^{\otimes\integers}$ for the APCRW) and, crucially, an underlying notion of density $\rho >0$. The coupling conditions \ref{pe:densitychange} - \ref{pe:nacelle} entail in a nutshell that, although the environment may mix slowly, this mixing in effect becomes faster when sprinkling in more particles thus slightly raising the density. 
 We will vary $\rho$ to observe its impact on the random walk, while fixing all the other parameters. 
 \\
 These conditions turn out to be quite versatile - for instance, they hold for any superposition of finitely many environments each satisfying them (up to innocuous changes, as detailed in Section~\ref{subsec:superpos}).
\\
\\
\textbf{The random walk.} \\
Now with the environment set up, we can define the discrete-time random walk $(X_n)_{n \geq 0}$ with two additional parameters $\pfull, \pempty \in [0,1]$, where $\pfull > \pempty$. The case $\pfull < \pempty$ can be readily treated in the same way to obtain symmetric results, and will be left out in the sequel.
We start the random walk from $X_0 = 0$. Then for every integer $n\geq 0$, we let $U_n$ be uniform over $[0,1]$, independent of everything else, and set
\begin{equation}\label{e:def-RWDRE}
X_{n+1}=X_n+2\times{\1}\big\{ U_{n} \leq (p_{\circ}-p_{\bullet})\1\{ \eta_{n +m}(X_n)=0 \} + p_{\bullet}\big\}-1.
\end{equation}
In words, if $\eta_n(X_n) = 0$ (meaning that the current site of the walker is ``empty", i.e. does not contain any environment particle), the walker jumps to the right with probability $\pempty$, and jumps to the left with probability $1 - \pempty$. 
Otherwise, if $\eta_n(X_n) \geq 1$ (i.e. the current site of the walker contains at least one particle, hence is deemed ``occupied"), the walk jumps to the right with probability $\pfull$ at time $n$, and jumps to the left with probability $1 - \pfull$. 
These dynamics are illustrated in \Cref{fig:RW Def}. 

\begin{figure}
    \centering
    \includegraphics[width=\linewidth]{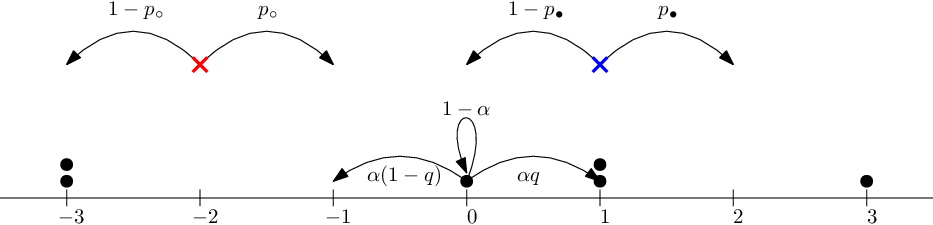}
    \caption{An illustration of a random walk on an Asymmetric Poisson Cloud of Random Walks. The bottom set of points represents the APCRW, and the top points represent the random walk on that APCRW environment. The arrows show the jump probabilities. The red walker (left cross) has no environment particles underneath it, so it uses $\pempty$, and the blue walker (right cross) does have environment particles underneath it, so it uses $\pfull$.}
    \label{fig:RW Def}
\end{figure}

\subsection{Results}
Our first result is a law of large numbers for the random walk on any environment satisfying our conditions, for most values of the density. 

\begin{The}\label{thm:moregeneral}
Let $\eta$ be an environment satisfying~\ref{pe:markov}-\ref{pe:density} and  \ref{pe:densitychange}-\ref{pe:nacelle}. Then for all $\rho \in (0,\infty)$ except at most countably many, there exists $v(\rho)$ such that 
\begin{equation}\label{eq:thmmoregeneral}
\frac{X_n}{n}   \underset{n\rightarrow \infty}{\longrightarrow} v(\rho) \,\,\text{ a.s.}
\end{equation}
Moreover, the map $\rho \mapsto v(\rho)$ is increasing. 
\end{The}
\begin{Rk}[Strict monotonicity]
The result holds if we only assume that~\ref{pe:markov}-~\ref{pe:density}, \ref{pe:drift} and~\ref{pe:couplings} hold, with the caveat that we only show that $\rho \mapsto v(\rho)$ is non-decreasing. 
\end{Rk}
The dependency of the speed $v(\rho)$ on the other parameters of the random walk is implicit, and throughout, we will write $v(\rho)$. 

Up to the at most countably many exceptional densities, this generalises \cite[Theorem 1.1]{HKT20}, which applied to the Symmetric Exclusion Process (SEP) and Poisson Cloud of Random Walks (PCRW, that is, the APCRW when $q=1/2$), two environments which satisfy the conditions of the theorem (as established in  \cite[Section 6 and Appendix B]{CKKR24}). 
\\

In the specific instance of the APCRW, we prove that the speed actually exists for all densities except at most one. 
It is currently an open question whether or not a law of large numbers holds 
at this critical density. 
\begin{The}\label{thm:main}
For the APCRW, with $\alpha, q, \pfull, \pempty \in (0,1)$ and $\pfull > \pempty$, for all $\rho \in (0,\infty)$ except possibly one value $\rho_0 >0$, there exists $v(\rho)\in (-1,1)$ such that
\begin{equation}\label{eq:thmmain}
\frac{X_n}{n}   \underset{n\rightarrow \infty}{\longrightarrow} v(\rho) \,\,\text{ a.s.}
\end{equation}
Moreover, the map $\rho \mapsto v(\rho)$ is increasing. 
\end{The}
\begin{Rks}\textit{Changing the activity rate of the APCRW.} 
Our results can be readily adapted when changing the activity $\nu\in (0, \infty)$ of the environment, namely the number of steps attented by any APCRW particle per unit of time. This requires 
a few benign (but cumbersome) adaptations in the proofs when $\nu,1/\nu \not \in \mathbb{N}$. For the sake of simplicity, we focus on the case $\nu=1$ in this paper. 
\\
\textit{Central limit theorem.} Our proof actually yields an (annealed) central limit theorem for all $\rho \neq \rho_0$, owing to the renewal times from Section~\ref{sec:onedensity} having all their moments finite. We leave out the details.
\end{Rks}

To illustrate the versatility of our approach, we state a generalisation of Theorem~\ref{thm:main} to an arbitrary (finite) superposition of APCRWs. Fix $N\in \mathbb{N}$ and let $\eta=\sum_{i=1}^N\eta^{(i)}$ where for each $i$, $\eta^{(i)}$ is an APCRW with parameters $\alpha_i,q_i\in (0,1)$ and density $\beta_i\rho$ with $\beta_i>0$, all the $\eta^{(i)}$'s being together independent. As before, the random walk in space-time position $(x,t)$ has probability $p_\bullet$ (resp. $p_\circ$) to jump to the right if $\eta_t(x)\geq 1$ (resp. $\eta_t(x)= 0$), and complementary probability to jump to the left. 

\begin{Cor}[Mixture of APCRWs] \label{cor:APCRWsum}
For all $N\in \mathbb{N}$, $(\alpha_i)_{1\leq i\leq N},(q_i)_{1\leq i\leq N}\in (0,1)^N$, $(\beta_i)_{1\leq i\leq N}\in (0,\infty)^N$ and $p_\bullet>p_\circ \in (0,1)$, the following holds.
For all $\rho \in (0,\infty)$ except at most countably many, the random walk $X$ on the mixture of APCRWs defined above satisfies
\begin{equation}\label{eq:thmmoregeneral}
\frac{X_n}{n}   \underset{n\rightarrow \infty}{\longrightarrow} v(\rho) \,\,\text{ a.s.}
\end{equation}
for some limit $v(\rho)$ such that the map $\rho \mapsto v(\rho)$ is increasing.

Moreover, if the quantity $d_i=\alpha_i(2q_i-1)$ does not depend on $i$, then the above conclusion holds for all $\rho$ except at most one value. 
\end{Cor}

\subsection{Proof structure}


\noindent \textbf{Proof of Theorem \ref{thm:moregeneral}.} 
We want to use a finite-range approximation of our random walk where we refresh the environment after a fixed time $L$ (and multiples thereof) by resampling from the initial distribution, independently of the environment's history up to time $L$ (as in~\cite{CKKR24}, itself inspired from similar a framework to tame long-range correlations, see e.g.~\cite{PhasetransGFF,SharpnessGFF,RI-III}). We will split our proof into two parts: finding a candidate limit $v(\rho)$ for the speed and then proving that we do indeed have that $X^\rho_n/n \rightarrow v(\rho)$ almost surely. We will informally write $X^{\rho}$ for the random walk in environment with density $\rho$, and $X^{\rho,L}$ when that environment is resampled every $L$ timesteps.

The finite-range models yield a natural prospect for the speed: for all $\rho>0$ and $L\geq 1$, if $X$ is a random walk under the $L$-range environment, one clearly has the a.s. convergence $\lim_{n \rightarrow \infty}\frac{X^{\rho,L}_n}{n}\underset{n\rightarrow\infty}{\longrightarrow}v(\rho,L)$ for an ad hoc constant $v(\rho,L)$, as the trajectory of $X$ consists of i.i.d. segments of $L$ steps. 
We want to consider $\lim_{L \rightarrow \infty} v(\rho,L)$ as our candidate, but it is not a priori evident that this limit exists. 
\\
In order to prove that it does, we need to compare different finite-ranged approximations by gradually sending $L$ to infinity. 
We can do this at the cost of slightly changing the density of the environment (Proposition~\ref{prop: many to one coupling}, itself a generalisation of~\cite[Lemma 4.1]{CKKR24}, illustrated in Figure~\ref{fig:many to one proof}): we can couple two environments $\eta,\eta'$ of densities say $\rho>\rho'$, and which are refreshed after times $L,L'$. For the first $\min(L,L')$ timesteps, we can have $\eta \succcurlyeq\eta'$ by~\ref{pe:monotonicity} (and readily couple walkers $X^{\rho,L},X^{\rho',L'}$ on these environments so that $X^{\rho',L'}$ does not overtake $X^{\rho,L}$ as $p_\bullet \geq p_\circ$ - see Proposition~\ref{lem:monotone} for generic deterministic monotonicity statements). The resampling at time $\min(L,L')$ of only one of the two environments makes us temporarily lose this domination. The ability to recover it (in a time typically $<<\min(L,L')$, so that the uncontrolled trajectory of the walkers in that interval does not affect the drift of $X^{\rho,L},X^{\rho',L'}$ at first order) is encapsulated in \ref{pe:couplings}. 
In substance, Proposition \ref{prop: many to one coupling} establishes that
\begin{equation} \label{eq: simplified many to one}
    v(\rho, L) \leq v\left(\rho + \beta(L, L'), L'\right) + \alpha(L,L')
\end{equation}
where $\alpha(L,L'), \beta(L, L') > 0$ converge to $0$ conveniently fast as $L, L' \rightarrow \infty$. 
\\
By repeatedly using \eqref{eq: simplified many to one} dyadically over $L,L'=2L,L''=4L...$, we can prove that
\begin{equation}\label{eq: limsup liminf}
    \limsup_{L \rightarrow \infty} v(\rho, L) \leq \liminf_{L \rightarrow \infty} v(\rho + \varepsilon, L)
\end{equation}
for all $\rho, \varepsilon > 0$. 
Once we have this, using that the ratio $X_n/n$ must always be in the interval $[-1,1]$, 
we can safely conclude that $\lim_{L \rightarrow \infty} v(\rho, L)$  must exist for all but countably many $\rho$, and we write $v(\rho)$ for the limit. 
\\

The second part of the proof is to establish the a.s. convergence $\frac{X^\rho_n}{n}\rightarrow v(\rho)$. 
To this end, we show that for any fixed $\delta>0$, the probabilities that the observed speed of the walk $\frac{X^\rho_n}{n}$ is not within distance $\delta$ of $v(\rho)$ are summable in $n$. This happens in the following three steps (say for the upper bound, the lower bound being similar): 
\begin{itemize}
    \item We compare $X^\rho_n$ to $X^{\rho + \varepsilon, L}_n$, where $\varepsilon$ is chosen such that ${v}(\rho + \varepsilon)$ exists, and small enough so that ${v}(\rho + \varepsilon) < {v}(\rho) + \delta/6$ (which is possible by the first part of the proof). Here, we take $n,L>>1$ so that we can use 
    Proposition \ref{prop: many to one coupling}. 
    Note that $X^{\rho}$ and $X^{\rho, n}$ coincide on the time-interval $[0,n]$. 
    \item We then compare $\frac{X^{\rho + \varepsilon, L}_n}{n}$ to its expectation (which is close to its speed $v(\rho+\varepsilon,L)$) using classical concentration bounds, taking advantage of the i.i.d. characteristics of the finite-ranged model. This requires to take $n >> L$ in an appropriate way (e.g. $L=\lfloor \sqrt{n}\rfloor$). 
    \item We conclude by the fact that $v(\rho+\varepsilon,L)\rightarrow v(\rho+\varepsilon)$ and our choice of $\varepsilon$.
\end{itemize}

\noindent \textbf{Proof of Theorem \ref{thm:main}.}
Here, we focus specifically on the APCRW environment. We first need to check that it satisfies the conditions for Theorem \ref{thm:moregeneral}, which we do in Appendix \ref{sec:couplingproofs}. The key technical input is to control the joint distribution of a cloud of independent walkers starting from given positions, and after a certain amount of time, via the soft local times framework~\cite{HHSST15,PopovTeixeira}, in Lemma~\ref{Lemma: Soft Local Times}. We need in passing to extend heat kernel bounds for the simple random walk to the asymmetric one (Lemma~\ref{Lemma: Heat Kernel}). 
Applying then Theorem \ref{thm:moregeneral}, what remains is to effectively fill in the at most countably many gaps of density, leaving at most one.
\\
This density will be the unique $\rho_0$ such that $v(\rho_0)$ (if it exists) would coincide with $(2q-1)\alpha$, the speed of the particles in the environment. Note that $\rho_0$ is unique indeed as $v$ is increasing by Theorem~\ref{thm:moregeneral}.
\\
The reason is that outside of $\rho_0$, we expect the random walk $X$ to travel at a significantly different speed from the particles of the environment, hence regularly refreshing its surroundings, which then allows us to implement a renewal structure. Namely, when picking a density away from $\rho_0$, say wlog $\rho >\rho_0$, then by \Cref{thm:moregeneral}, we can find $\rho'\in (\rho_0,\rho)$ such that $v({\rho'})>(2q-1)\alpha$ and $X^\rho$ dominates stochastically $X^{\rho'}$, itself faster than the environment.
\\
This ballistic separation is formalised in Corollary~\ref{cor: ballisticity}, which in the current context translates as
\begin{equation}\label{eq: ballisticity req}
\dP \left( \exists n \in \dN : X^\rho_n < n v(\rho') - K \right) \leq \exp\left( -K^{c} \right)
\end{equation}
for all $K$ large enough, and an explicit constant $c>0$. Informally, this means that up to a lag with subpolynomial tails, the random walk will go faster than the particles surrounding it (which travel at speed $(2q-1)\alpha <v(\rho')$, with large deviation controls). To construct a renewal time, we wait for $X$ to make a large number of consecutive steps to the right (up to logarithmic in time is likely to happen, by uniform ellipticity), to offset the aforementioned lag. To handle the technicalities, we can transpose the renewal structure from~\cite[Section 4]{HHSST15} for the PCRW.
\\
We insist that the ballisticity requirement~\eqref{eq: ballisticity req} (playing the role of (1.7) in \cite{HHSST15}) is what prevents us from directly building this renewal structure without first proving Theorem \ref{thm:moregeneral}.
\\

\noindent\textbf{Proof of Corollary~\ref{cor:APCRWsum}.} This follows from the previous two theorems and the remark that our conditions~\ref{pe:markov}-\ref{pe:density} and~\ref{pe:densitychange}-\ref{pe:nacelle} hold for any superposition of independent environments satisfying those (Proposition~\ref{prop: apcrw superposition}), with harmless changes to constants. One can then check readily that our arguments in the proof of Theorem~\ref{thm:moregeneral} (as well as those of~\cite{CKKR24} used for the monotonicity of $v$) still apply. In the case that all the drifts $d_i$ are equal, one checks easily that the renewal structure used for Theorem~\ref{thm:main} is still relevant.

\subsection{Related works and discussion}\label{subsec:discussion}
The area of random walks on dynamic environments, especially particle systems, has been very active over the past ten to fifteen years. 
\\
\textbf{Relaxing the conditions on the environment.} Until recently, many works dealt with environments that were more amenable to renewal structures than clouds of random walks: non-conservative systems (such as the contact process~\cite{Allasia,MountfordVares}), or more generally better mixing conditions~\cite{AvenaBlondelFaggionato,BaldassoFluctuation,BethuelsenVollering,BlondelHilarioTeixeira}. By contrast, most results on the SEP or the PCRW obtained so far were perturbative in at least one parameter, whether the density of particles~\cite{RWonRWslowdensity,RWonRWshighdim,HHSST15}, or their activity rate~\cite{HS15,SalviSimenhaus}. 
\\
Part of the reason was to obtain ballisticity estimates on the walker to ensure that it would travel strictly faster (or slower) than the environment~\cite{BallisticRWsLateralDecoupling}. Building on this assumption, called the \textit{non-nestling} case, renewal structures (e.g. in~\cite{HHSST15} for the PCRW, in~\cite{HS15} for the SEP) would imply the law of large numbers. We stress that this ballisticity condition is in general highly non-trivial to obtain. For instance, the assumption (2.8) in~\cite{BallisticRWsLateralDecoupling} corresponds to our Corollary~\ref{cor: ballisticity}, for which we need to deploy the full machinery from the proof of Theorem~\ref{thm:moregeneral} described in the previous section. 
\\
The purpose of the current work is also to enhance the relevance of the general approach of~\cite{CKKR24}, which treated any environment satisfying the mild regularity and mixing conditions~\ref{pe:markov}-\ref{pe:density} and~\ref{pe:densitychange}-\ref{pe:nacelle}, via approximation schemes with finite-range models. In~\cite{CKKR24}, it was shown that $v$ is monotonic in $\rho$, provided it existed. We complement this by establishing the existence of $v$, which was so far only done for the SEP and PCRW, via very different methods~\cite{HKT20}. While one may hope to adapt this latter work for the APCRW (turning the lateral decoupling of their section 4 into a 'diagonal decoupling' along the drift of the environment particles - which was 0 for the PCRW), this is much less clear for instance in the case of a mixture of APCRWs with different drifts, while our method adapts seamlessly (up to sacrificing a null measure set of densities) in Corollary~\ref{cor:APCRWsum}. 
\\
\textbf{The speed picture.} Going back to the case of a single drift for the APCRW, the picture we obtain for $\rho \mapsto v(\rho)$ is overall the same as that derived for the PCRW/SEP when combining~\cite{CKKR24}-\cite{HKT20}. Namely, $v$ is increasing (for $p_\bullet>p_\circ$ and defined everywhere except possibly when crossing the horizontal line $y=(2q-1)\alpha$. A natural question is whether $v$ is continuous. This seems achievable by now standard means outside of the critical point (e.g. via the aforementioned renewal strategies or renormalising structures taking advantage of the speed difference between the walker and the particles). Whether $v$ is continuous, or even defined at the critical point is a much more delicate matter. The exploitation of monotonicity to generate regularity in a `boomerang' way as we do in this work does not appear yet to bring the key to this lock. 
\\
\textbf{Other environments: towards the ASEP.} Finally, we make here an important step towards establishing an analogous LLN for the random walk driven by the ASEP in the general case (the non-nestling case being treated in~\cite{BallisticRWsLateralDecoupling}): if we can show that it satisfies our conditions, this will now automatically entail the existence of $v$. It is however far from immediate to show density preservation and coupling conditions such as~\ref{pe:densitychange} and~\ref{pe:couplings} for the ASEP, which is a noticeably more complicated object than a 'SEP+drift'. In particular, it is non-reversible and does not enjoy a more workable reformulation as an interchange process. In addition, reducing the exceptional densities from countably many to at most one (or two) is much less clear: a renewal structure as in our Section~\ref{sec:onedensity} would be more delicate to implement if the walker is slower than the particles, as their interdependencies a priori extend the range of correlated space-time zones. 

\subsection{Plan of the paper}
In Section~\ref{sec:defs}, we define the generic class of dynamic environments we will be working with as well as the corresponding random walk. We state the properties~\ref{pe:markov}-~\ref{pe:density} and the coupling conditions~\ref{pe:densitychange}-\ref{pe:nacelle} for our environments, which are established in Appendix~\ref{sec:couplingproofs} for the APCRW. We also formalise the superposition of environments (with Corollary~\ref{cor:APCRWsum} in mind), and give elementary monotonicity properties that will be useful throughout the paper. In Section~\ref{sec:couplings}, we introduce the finite range model and establish Proposition~\ref{prop: many to one coupling}. In  Section~\ref{sec:speed}, we prove Theorem \ref{thm:moregeneral}, and in Section~\ref{sec:onedensity}, we prove Theorem~\ref{thm:main} and Corollary~\ref{cor:APCRWsum}. 

\subsection{Notation} \label{subsec:not} 
For two integers $a\leq b$, we will denote by $[a,b]$ the set of integers $\{a,\ldots, b\}$ and declare that the length of $[a,b]$ is its number of elements, that is $b-a+1$. Non-numbered constants such as $c,c',\dots$ and $C,C',\dots$ are generic, purely numerical and can change from place to place. By contrast, numbered constants are fixed upon first appearance. We usually track their dependency in the parameters, but keep the dependency on the environment model itself (e.g. SEP, PCRW, etc.) implicit.

We will write $\geq_{\text{st.}}$ and $\le_{\text{st.}}$ the usual stochastic dominations for probability distributions.

\section{General model: conditions and basic properties}\label{sec:defs}
We describe the class of environments we will be working with in Section~\ref{subsec:DefRWRDREgeneral}, including the properties \ref{pe:markov}-\ref{pe:density}, and define the different annealed/quenched probabilities we will be using. 

\subsection{Definition of the general model}\label{subsec:DefRWRDREgeneral}
\textbf{{A class of dynamic random environments.}}
\\
The environments we consider are Markov processes taking values in $ \Sigma: = (\mathbb{Z}^+_0)^{\mathbb{Z}}$. 
For $I\subseteq \mathbb{Z}$ and $\eta \in \Sigma$, we write $\eta\vert_I$ for $\eta$ restricted to $I$, and when $I$ is finite, we set $\eta(I):=\sum_{x\in I}\eta(x)$. 
For $\eta, \eta' \in \Sigma$ and $I$ as before, we write  $\eta\vert_I\preccurlyeq \eta'\vert_I$ (or $\eta'\vert_I\succcurlyeq \eta\vert_I$) whenever $\eta(x)\le \eta'(x)$ for all $x \in I$ (we simply write $\eta\preccurlyeq \eta'$ or $\eta'\succcurlyeq \eta$ when $I=\mathbb{Z}$). 

An \textit{environment} is given by a non-empty bounded open interval $J\subseteq \mathbb{R_+}$ 
and two families of probability measures: $({\mu_\rho} : \rho \in J)$, where $\mu_\rho$ is a measure on $\Sigma$, and $(\mathbf{P}^{\eta_0}: \eta_0 \in \Sigma)$. It is required to satisfy the following conditions, which are the same as in~\cite{CKKR24} (the only difference being the removal of the symmetry condition in~\ref{pe:markov}).
\begin{align}
&\tag{P.1}\label{pe:markov} \text{\parbox{14cm}{\textit{(Markov property and invariance).} For every $\eta_0\in \Sigma$, the process $(\eta_t)_{t \geq 0}$ defined under $\mathbf{P}^{\eta_0}$ is a time-homogeneous Markov process, such that for all $s\geq 0$ and $x\in \mathbb{Z}$, the process $(\eta_{s+t}(x+\cdot))_{t\geq 0}$ has law $\mathbf{P}^{\eta_s(x+\cdot)}$.
}}
\end{align}
\begin{align}
&\tag{P.2} \label{pe:stationary} \text{\parbox{14cm}{\textit{(Stationary measure).} The initial distribution ${\mu}_\rho$ is a stationary distribution for the Markov process $\mathbf{P}^{\eta_0}$. More precisely, letting $${\mathbf{P}}^{\rho}= \int {\mu}_{\rho}(d\eta_0) \mathbf{P}^{\eta_0}, \quad \rho\in J,$$ 
the ${\mathbf{P}}^\rho$-law of $(\eta_{t+s})_{t\ge0}$ is identical to the ${\mathbf{P}}^\rho$-law of $(\eta_{t})_{t\ge0}$ for all $s\geq 0$ and $\rho \in J$. In particular, the marginal law of $\eta_t$ under ${\mathbf{P}}^\rho$ is ${\mu}_\rho$ for all $t \geq 0$. 
}}\\[0.5em]
&\tag{P.3} \label{pe:monotonicity} \text{\parbox{14cm}{\textit{(Monotonicity).}  The following stochastic dominations hold:
\begin{itemize}
\item[i)] (quenched) For all $\eta'_0 \preccurlyeq \eta_0$, one has that $\mathbf{P}^{\eta'_0} \leq_{\text{st.}}\mathbf{P}^{\eta_0}$ , i.e.~there exists a coupling of $(\eta'_t)_{t \geq0}$ and $(\eta_t)_{t \geq 0}$ such that  $\eta'_t\preccurlyeq  \eta_t$ for all $t\geq 0$. 
 \item[ii)] (annealed) For all $\rho'\leq \rho$, one has that ${\mu}_{\rho'} \leq_{\text{st.}}{\mu}_{\rho}$. Together with i), this implies that ${\mathbf{P}}^{\rho'} \leq_{\text{st.}}{\mathbf{P}}^{\rho}$.
\end{itemize}
}}\\[0.5em]
&\tag{P.4} \label{pe:density} \text{\parbox{14cm}{\textit{(Density at stationarity).} There exists a constant $\Cl[c]{densitydev} >0$  such that for all $\rho\in J$, for all $\varepsilon \in (0,1)$ and every positive integer $\ell$, 
$$
\mathbf{P}^\rho(\{\eta_0\in \Sigma: \,  \vert\eta_0([0,\ell-1 ])-\rho\ell\vert \geq \varepsilon \ell \})\leq 2\exp(-\Cr{densitydev}\varepsilon^2\ell).
$$}}
\end{align}



\noindent
\textbf{The random walk.}
\\
Recall that for a given realisation $\eta=(\eta_t)_{t\geq 0}$ of the environment, we define the random walk $X=(X_n)_{n\geq 0}$ as in~\eqref{e:def-RWDRE}, We now allow it to start from any initial space-time position $z=(x,m)\in \mathbb{Z}\times \mathbb{Z}^+_0$ 
In words, when $X_n$ is on an occupied (resp.~empty) site of $\eta_n$,  it jumps to its right neighbour with probability $p_{\bullet}$ (resp.~$p_{\circ}$) and to its left neighbour with probability $1-p_{\bullet}$ (resp.~$1-p_{\circ})$. 

We call $P^{\eta}_z$ the \textit{quenched} law of the walk started at $z$ (and abbreviate $P^{\eta}=P^{\eta}_{(0,0)}$), in the sense that the environment $\eta$ is deterministic.
\\
When averaging over the randomness of the environment, we define two annealed measures.  Recall
that $\mathbf{P}^{\eta_0}$ denotes the law of $\eta$ with fixed initial configuration $\eta_0 \in \Sigma$ and that $\mathbf{P}^{\rho}$ describes the dynamics with random initial configuration sampled against the stationary measure $\mu_\rho$. We then let
\begin{align}
& \label{eq:RW_an}  \mathbb{P}^{\eta_0}_{z}[\, \cdot \, ]=\int \mathbf{P}^{\eta_0}(d\eta)P^{\eta}_z[\, \cdot \, ],\quad \eta_0 \in \Sigma,\\
& \label{eq:RW_ann} \mathbb{P}^{\rho}_{z}[\, \cdot \, ]=\int \mathbf{P}^{\rho}(d\eta)P^{\eta}_z[\, \cdot \, ], \quad \rho \in J, 
\end{align} 
for arbitrary $z \in  \mathbb{Z}\times \mathbb{Z}^+_0$. 
Observe that
$\mathbb{P}^{\rho}_{z} =  \int {\mu}_{\rho}(d\eta_0) \mathbb{P}^{\eta_0}_z$.  For $\star\in\{\eta,\rho\}$, write $\dP^{\star}_x$ for $\dP^{\star}_{(x,0)}$, for all $x\in \mathbb{Z}$, and write simply $\dP^\star$ for $\dP^{\star}_0$. 

\subsection{Conditions on the model}\label{subsec:conditions}

We state here the conditions from \cite{CKKR24}, with a slight modification in~\ref{pe:nacelle} to accommodate the drift of the environment. They are satisfied by the APCRW (see Proposition~\ref{prop:conditionsforAPCRW}, whose proof is postponed to the Appendix~\ref{sec:couplingproofs}). We assume that there exists a fixed constant $\nu>0$ such that the following conditions hold (as mentioned in the introduction, $\nu$ should be understood as the intrisic activity rate of the environment, which we will take $=1$ in the case of the APCRW).



\begin{enumerate}[label=(C.\arabic*) ]
\item \label{pe:densitychange} 
\textit{(Conservation of density).} There exists constants $\Cl{densitystable},\Cl[c]{densitystableexpo} \in (0,\infty)$ such that for all $\rho \in J$ and $\varepsilon \in (0,1)$ (with $\rho+\varepsilon \in J$), and for all 
$\ell, \ell', H,t\geq 1$ satisfying $H>4{\nu} t > \Cr{densitystable}\ell^2\varepsilon^{-2}(1+{ \vert\log^3 (\nu t)}\vert)$ and $\ell'\leq \sqrt{t}$, the following two inequalities hold. Let $\eta_0$ be such that on every interval $I$ of length $\ell$ included $[-H,H]$, one has $ \eta_0(I)\leq (\rho+\varepsilon)\ell$ (resp.~$ \eta_0(I)\geq (\rho-\varepsilon)\ell$).  Then
\begin{equation} \label{eq:densitychange}
{\mathbf{P}}^{\eta_0} \left(\begin{array}{c} \text{for all intervals $I'$ of length $\ell'$ included in} \\ \text{$[-H+2{\nu} t,H-2{\nu} t]$: $\eta_t(I') \leq (\rho+3\varepsilon) \ell'$}\\ \text{(resp.~$\eta_t(I')  
 \geq (\rho-3\varepsilon) \ell'$)} \end{array} \right) \geq 1 - 4  H \exp(-\Cr{densitystableexpo}\varepsilon^2\ell').
\end{equation}
\item\label{pe:compatible} \textit{(Couplings).} There exists $\Cl{compatible}, \Cl{SEPcoupling2}\in (0,\infty)$ such that the following holds. Let $\rho \in J$, $\varepsilon \in (0,1)$ (with $\rho+\varepsilon \in J$), and $H_1, H_2,t,\ell\geq 1$ be integers { such that $\min\{H_1, H_2-H_1-1\}> 10{\nu}t>4\nu\ell^{100}>\Cr{compatible}$, $\nu^{} \ell>\Cr{compatible}\varepsilon^{-2}(1+\vert\log^3(\nu\ell^4)\vert)$ and $\ell>80\nu\varepsilon^{-1}+\nu^{-2}$}. Let $\eta_0,\eta'_0\in \Sigma$ be such that $\eta_0\vert_{[-H_1, H_1]}\succcurlyeq \eta_{0}'\vert_{[-H_1, H_1]}$ and such that for every interval $I\subseteq [-H_2,H_2]$ of length $\lfloor \ell/2 \rfloor \leq |I| \leq \ell$ , we have $\eta_0(I)\geq (\rho+3\varepsilon/4)\vert I \vert$ and $\eta'_0(I)\leq (\rho+\varepsilon/4)\vert I \vert$. Then there exists a coupling $\dQ$ of two environments $\eta,\eta'$ with respective marginals $\mathbf{P}^{\eta_0},\mathbf{P}^{\eta'_0}$ such that 
\begin{equation}\label{eq:compatible1}
\dQ\big(\forall s\in [0,t], \, \eta_s\vert_{[-H_1+4{\nu}t, H_1-4{\nu}t]}\succcurlyeq \eta'_s\vert_{[-H_1+4{\nu}t, H_1-4{\nu}t]} \big)\geq 1 - 20t\exp (-{\nu}t/4)
\end{equation}
and 
\begin{equation}\label{eq:compatible2}
\dQ\big(\eta_t \vert_{[-H_2+6{\nu}t, H_2-6{\nu}t]} \succcurlyeq \eta'_t\vert_{[-H_2+6{\nu}t, H_2-6{\nu}t]} \big)\geq 1 -5\Cr{SEPcoupling2} \ell^4 H_2\exp\left(- \Cr{SEPcoupling2}^{-1}\textstyle\frac{\nu}{\nu+1}\varepsilon^2\ell\right).
\end{equation}
\end{enumerate}


The next two coupling conditions are each part of \ref{pe:compatible}.
The first of the two, \ref{pe:drift}, allows us to take a pair of environments where one covers the other over a given interval and maintains this monotonicity in the environments, losing a little off the ends of the interval (to account for external particles potentially drifting in). 
By covering, we mean that every site at which the lower density environment has a particle, the higher density environment will have at least as many particles there.

The second, \ref{pe:couplings}, allows us to take two environments where one has a slightly higher empirical density than the other in an area and couple them such that after a relatively short time, the environment with the higher empirical density covers the environment with the lower empirical density (in that area). 

When we combine these in \ref{pe:compatible}, we maintain an interval where one environment covers the other in the centre (as in \ref{pe:drift}), while expanding this domination on neighbouring intervals (as in \ref{pe:couplings}).

\begin{enumerate}[label=(C.2.\arabic*)]
\item\label{pe:drift} \textit{(No particle drifting in from the side).} 
Let $t,H \geq 0$ and $k\geq 1$ be integers, and let $\eta_0,\eta'_0\in \Sigma$ be such that $\eta_0 \vert_{[-H, H]}\succcurlyeq \eta_0' \vert_{[-H, H]}$.   
There exists a coupling $\dQ$ of environments $\eta,\eta'$ with respective marginals $\mathbf{P}^{\eta_0}$ and $\mathbf{P}^{\eta'_0}$ such that
\begin{equation}\label{eq:SEPdriftdeviations}
\mathbb{Q}\big(\forall s\in [0,t], \, \eta_s\vert_{[-H+2{\nu}kt, H-2{\nu}kt]}\succcurlyeq\eta'_s\vert_{[-H+2{\nu}kt, H-2{\nu}kt]} \big)\geq 1- 20\exp(-k{\nu} t/4).
\end{equation}

\item\label{pe:couplings} \textit{(Covering $\eta'$ by $\eta$).}
There exists $\Cl{SEPcoupling} >0$ such that for all $\rho\in J$ and $\varepsilon\in (0,1)$ (with $\rho+\varepsilon \in J$), the following holds. If  $H,t\geq 1$ satisfy {$H>4{\nu} t$, $\nu^8t>1$ and $\nu t^{1/4} >\Cr{SEPcoupling}\varepsilon^{-2}(1+\vert \log^{3}(\nu t) \vert)$}, then for all $\eta_0,\eta'_0\in \Sigma$ such that on each interval $I \subset [-H,H]$ of length $\lfloor \ell/2 \rfloor \leq |I| \leq \ell$, where $\ell:=\lfloor t^{1/4}\rfloor$, $\eta_0(I)\geq (\rho+3\varepsilon/4)|I|$ and $ \eta'_0(I) \leq (\rho+\varepsilon/4)|I|$, there exists a coupling $\mathbb{Q}$ of $\eta,\eta'$  with marginals $\mathbf{P}^{\eta_0}, \mathbf{P}^{\eta'_0}$ so that \begin{equation}\label{eq:doublecoupleSEP}
\mathbb{Q}( \eta_t \vert _{[-H+4{\nu} t, H-4{\nu} t]}\succcurlyeq \eta_t' \vert _{[-H+4{\nu} t, H- 4{\nu} t]})\geq 1- \Cr{SEPcoupling2} tH\exp\big(-(\Cr{SEPcoupling2}(1+\nu^{-1}))^{-1}\varepsilon^2t^{1/4} \big).
\end{equation}

\end{enumerate}


The third of the coupling conditions ensures that when we slightly increase the density of an environment, a random walker is able to encounter points at which the environment is empty for the lower density environment, but occupied for the higher density environment. We will then be able to separate two random walks under either environment at that point. For simplicity's sake, we are less explicit and more general than~\cite{CKKR24} in tracking the interdependency of the constants in this condition, and we will readily check that the arguments we use from~\cite{CKKR24} for the strict monotonicity of $v$ still hold with this version of~\ref{pe:nacelle}. 

\begin{enumerate}[label=(C.\arabic*) ]
\setcounter{enumi}{2}
\item\label{pe:nacelle} \textit{(Sprinkler).} 
There exist positive constants $\Cl{sprinkler1}=\Cr{sprinkler1}(\nu), \Cl{sprinkler2}=\Cr{sprinkler2}(\Cr{sprinkler1}, J, p_\circ, p_\bullet)$ such that the following holds. For all $\rho \in J$, for all integers $H, \ell,k\geq 1$ with $H\geq 2\nu\ell k$ and 
$k \geq \Cr{sprinkler2}$, the following holds. Let $I= [0,\ell]$. If $\eta_0,\eta'_0 \in \Sigma$ satisfy $\eta_0(x)\geq \eta'_0(x)$ for all $x\in [-H,H]$, $ \eta_0(I)\geq \eta'_0(I)+1$, and $\eta'_0([-3\ell+1, 3\ell])\leq 6(\rho+1)\ell$,  then for any $x \in \{0, 1\},$  there is a coupling $\mathbb{Q}$ of  $\eta'$ under ${\mathbf{P}}^{\eta_0'}$ and $\eta$ under ${\mathbf{P}}^{\eta_0}$ such that, with 
$
\delta = \Cr{sprinkler1} ^{6(\rho+1)\ell}, 
$
\begin{equation}\label{eq:penacelleNEW}
\begin{split}
&\mathbb{Q}(\eta_{\ell}(x )>0  , \,    \eta'_{\ell}( x )=0)\geq 2\delta,
\end{split}
\end{equation}
and, writing $\delta'= \delta (p_{\circ}(1-p_{\bullet}))^{6(\rho+1)\ell}$,
\begin{equation}\label{eq:penacelleNEW2}
\begin{split}
\mathbb{Q}(\{\forall s\in [0, \ell],  \, \eta_s\vert_{[-H+2{\nu}k\ell, H-2{\nu}k\ell]}\succcurlyeq\eta'_s\vert_{[-H+2{\nu}k\ell, H-2{\nu}k\ell]}\}^c)&\leq 20e^{-k\nu \ell/4}
 \leq \delta'.
\end{split}
\end{equation}
\end{enumerate}

\subsection{Superposition of environments}\label{subsec:superpos}



We will now define the generic superposition of $N\in \mathbb{N}$ environments (that we apply to the APCRW in Corollary \ref{cor:APCRWsum}), and adjust the conditions \ref{pe:markov}-\ref{pe:density} and \ref{pe:densitychange}-\ref{pe:nacelle} accordingly.
\\
We set $J = (a_1, b_1)\times (a_2, b_2) \times \dots \times (a_N, b_N) \subset \dR^n_+$, with $0< a_i < b_i$ for all $i$. For each $i$, let $\eta^{(i)}$ be an environment with associated invariant measures $({\mu^{(i)}_{\rho_i}} : \rho_i \in (a_i,b_i))$ on $\Sigma$ and evolution measures $(\mathbf{P}^{(i),\eta^{(i)}_0}: \eta^{(i)}_0 \in \Sigma)$. 
For every $\rho = (\rho_1, \rho_2, \dots, \rho_N) \in J$, we define the superposition $\eta=\left(\eta^{(1)}, \eta^{(2)}, \dots, \eta^{(N)}\right)$ pointwise, i.e. 
\begin{equation} \label{eq: envt superposition def}
    \eta_t(x)=\left(\eta^{(1)}_t(x), \eta^{(2)}_t(x), \dots, \eta^{(N)}_t(x)\right)
\end{equation} 
for every time $t\geq 0$ and position $x\in \mathbb{Z}$. Hence $\eta$ is a Markov process taking values in $\Sigma^{(N)} := \left(\left(\dZ^+_0\right)^N\right)^{\dZ}$, which can be thought as having $\eta^{(i)}_t(x)$ particles of type $i$ at $(x,t)$. 
Given this environment, we define the random walk as in Section \ref{sec: model}, with $\left\{\eta_{n+m}(X_n) = 0\right\}$ in  \eqref{e:def-RWDRE} being replaced by $\left\{\eta^{(i)}_{n+m}(X_n) = 0\right\}$ for all $1\leq i \leq N$.
In other words, the random walker uses jump probabilities $\pfull$ and $1-\pfull$ if any $\eta^{(i)}$ is occupied at its current space-time position, and $\pempty$ and $1-\pempty$ otherwise. 
\\
\\
Recall that as \textit{environments}, each $\eta^{(i)}$ for $1\leq i \leq N$ follows \ref{pe:markov}-\ref{pe:density}. Note that we can pick $\Cr{densitydev}>0$ such that the inequality in \ref{pe:density} holds simultaneously for every $\eta^{(i)}$. We now state an adapted version of the conditions~\ref{pe:densitychange}-\ref{pe:nacelle} for the superposed environment $\eta$.
\\ 
For each of~\ref{pe:densitychange},~\ref{pe:compatible},~\ref{pe:drift} and~\ref{pe:couplings}, we will require the condition to hold for every $\eta^{(i)}$, with bespoke uniformisation of the constants over $1\leq i \leq N$. 
\\
For~\ref{pe:nacelle}, we require $H\geq 2\nu_i \ell k$, $\eta^{(i)}_0(x)\geq \eta'^{(i)}_0(x)$ for $x \in [-H,H]$ and  $\eta'^{(i)}_0([-3\ell + 1, 3\ell]) \leq 6(\rho_i+1)\ell$ for all $i$. Additionally, we require  $\eta^{(i)}_0 \geq \eta'^{(i)}_0(I) + 1$ for every $1 \leq i \leq N$. 
We update \eqref{eq:penacelleNEW} to 
\begin{equation}\label{eq:penacelleNEW-N}
  \mathbb{Q}\left( \left\{ \exists i \in \{1, \dots, N\} : \eta_\ell^{(i)}(x)>0 \right\}\bigcap \left\{ \forall i \in \{1, \dots, N\},\eta'^{(i)}_\ell(x)=0 \right\}\right)\geq 2 \delta  
\end{equation}
with $\delta = \Cr{sprinkler1}^{6 \left(\sum_{i = 1}^N \left(\rho_i +1\right)\right)\ell}$, and~\eqref{eq:penacelleNEW2} to 
\begin{equation}\label{eq:penacelleNEW2-N}
\begin{split}
\mathbb{Q}\left(\bigcup_{i=1}^N\{\forall s\in [0, \ell],  \, \eta_s\vert_{[-H+2{\nu_i}k\ell, H-2{\nu_i}k\ell]}\succcurlyeq\eta'_s\vert_{[-H+2{\nu_i}k\ell, H-2{\nu_i}k\ell]}\}^c\right)&\leq 20Ne^{-k\min_{i}\nu_i \ell/4}
 \\
 &\leq \delta',
\end{split}
\end{equation}
keeping $\delta'= \delta (p_{\circ}(1-p_{\bullet}))^{6(\rho+1)\ell}$. 

This modification of~\ref{pe:nacelle} is tailored for the case when the $\eta^{(i)}$ are independent, as per the statement below. 
\begin{proposition} \label{prop: apcrw superposition} \label{prop: envt superposition}
    Let $\eta^{(1)}, \eta^{(2)}, \dots, \eta^{(N)}$ be independent environments in the sense of Sections \ref{subsec:DefRWRDREgeneral} and \ref{subsec:conditions}, such that~\ref{pe:nacelle} holds for each of them. Then~$\eta=(\eta^{(1)}, \ldots, \eta^{(N)})$ satisfies the new version of~\ref{pe:nacelle} for superposed environments. 
\end{proposition}
The proof is straightforward, noticing for $E$ being the event in~\eqref{eq:penacelleNEW-N} that $E$ contains the intersection of $\left\{\eta^{(i)}_\ell(x) >0, \eta'^{(i)}_\ell(x) = 0\right\}$ for all $1\leq i \leq N$, which is the corresponding event in~\eqref{eq:penacelleNEW} for every $\eta^{(i)}$. For~\eqref{eq:penacelleNEW2-N}, the first inequality holds by a union bound, and the second is satisfied as soon as $\Cr{sprinkler2}$ (and thus $k$) is larger than $4(\log (20N)+6\sum_{i}(b_i+1))/\min_i\nu_i$. 
Details are left to the reader. 


\subsection{Elementary monotonicity couplings}
We describe a joint construction for the walk $X$ when started at different space-time points, which is given in \cite[Section 2]{CKKR24}, and is itself a simplified version of the construction in in \cite[Section 3]{HKT20}. The aim  of this section is to state the generic weak monotonicity coupling in the Lemma~\ref{lem:monotone} below.

For a point $w=(x,n)\in\mathbb{Z} \times \mathbb{Z}^+_0$, we let $\pi_1(w)=x$ and $\pi_2(w)=n$ denote the projection onto the first (spatial) and second (temporal) coordinate. We define the discrete lattice
\begin{equation}
\label{e:def_space_time}
\mathbb{L}:= (2\mathbb{Z} \times 2\mathbb{Z}^+_0) \cup  \big( (1,1)+(2\mathbb{Z}  \times 2\mathbb{Z}^+_0)\big).
\end{equation}
For parity reasons, if $X=(X_n)_{n \geq 0}$ is started at $z=(0,0)$ under the quenched measure $P_z^{\eta}$ or any of the annealed measures at \eqref{eq:RW_an}-\eqref{eq:RW_ann}, $X$ will never leave the lattice $\mathbb{L} \subset (\mathbb{Z} \times \mathbb{Z}^+_0)$. 

Fix a realisation of the environment $\eta$. We define a family of processes $(X^w=(X^w_n)_{n \geq 0} : w\in\mathbb{L})$ starting simultaneously from all points of $\mathbb{L}$, such that each of them is a random walk on the environment $\eta$, and so that any two trajectories that ever meet at a space-time point evolve together at all future times. Concretely, we set $X^w_0=\pi_1(w)$. 

We let $U= \left(U_w\right)_{w\in\mathbb{L}}$ be a collection of i.i.d.~uniform random variables on $[0,1]$. We define a field $A=(A_w)_{w\in \mathbb{L}} \in \{-1,1\}^{\mathbb{L}}$ of {\it arrows}, 
measurably in $(\eta,U )$ as follows:
\begin{align}\label{def:A}
A_w = A(\eta_n(x), U_w) =2\times{\1}\big\{ U_{w} \leq (p_{\circ}-p_{\bullet})\1\{ \eta_n(x)= 0 \} + p_{\bullet}\big\}-1, \quad w=(x,n)\in\mathbb{L}.
\end{align}
For all  $w=(x,n)\in \mathbb{L}$, we let $X^w_0=x$ and, for all integer $ k \ge0$, we define recursively
\begin{align} \label{eq:defX}
X^w_{k+1} = X^w_k+A_{(X^w_k, \,  n+k)}.
\end{align}

By comparing \eqref{e:def-RWDRE} with \eqref{def:A}-\eqref{eq:defX}, it is straightforward that the law of $X^{(0,0)}= (X^{(0,0)}_n)_{n \geq 0}$ when averaging over $U$ is the same as that of $X$ under $P^{\eta}_{(0,0)}= P^\eta$. 

We conclude this section by recalling the useful monotonicity property~\cite[Lemma 2.2]{CKKR24} 
(see also \cite[Proposition 3.1]{HKT20}). It merely combines the facts that two trajectories on the same environment cannot cross without first meeting at a vertex, after which they merge and stay together, and that a random walk starting in some environment $\eta$ will stay to the left of a trajectory starting from the same point in an environment $\widetilde{\eta}$ such that $\widetilde{\eta}_t(x)\succcurlyeq \eta_t(x)$ at all relevant space-time points $(x,t)$. 

\begin{lemma}\emph{\cite[Lemma 2.2]{CKKR24}}\label{lem:monotone}
If $\eta,\widetilde{\eta} \in \Sigma$ and $K\subseteq \mathbb{Z}\times\mathbb{Z}^+_0$ are such that 
\begin{align}\label{domineta}
\eta_n(x)\le \widetilde{\eta}_n(x)\text{, for all } (x,n)\in\mathbb{L}\cap K,
\end{align}
then for every $w,w'\in K$ with $\pi_1(w')\le \pi_1(w)$ and $\pi_2(w)= \pi_2(w')$, and for every $n \geq 0$ such that $[\pi_1(w)-k,  \pi_1(w')+k]\times [\pi_2(w), \pi_2(w)+k] \subseteq K $ for all $0 \leq k \leq n$\color{black}, one has that
\begin{equation}\label{e:monotone1}
  X^{w'}_{n}\le \widetilde{X}^{w}_{n},
\end{equation}  
where $\widetilde{X}^w = (\widetilde{X}_n^w)_{n \geq 0}$ is defined as in \eqref{eq:defX}, with $\widetilde{\eta}$ in lieu of $\eta$ in \eqref{def:A}. 
\end{lemma}


\section{Coupling with a finite-range version of the model}\label{sec:couplings}
We introduce a finite range version of our environments in which the environment gets resampled according to $\mu_\rho$ after every time $kL$, $k\in \mathbb{N}$, independently of everything else (see~\cite[Section 3.3]{CKKR24} for details). We also state and prove a generalisation of \cite[Lemma 4.1]{CKKR24} (Proposition~\ref{prop: many to one coupling} in the following subsection), which follows from these conditions and will be instrumental in proving the law of large numbers in Section~\ref{sec:speed}.
\\
\\
Briefly, the finite range version of the model is defined as follows: for $\rho\in J$ and $L \in \mathbb{N}$,  we define a probability measure $\mathbf{P}^{\rho,L}$ (with expectation $\mathbf{E}^{\rho,L}$) on the set of environments $\eta=(\eta_t)_{t\geq 0}$ such that at every time $t$ multiple of $L$, $\eta_t$ is sampled according to ${\mu}_\rho$ (see \eqref{pe:stationary}), independently of $(\eta _s)_{0\leq s<t}$ and, given $\eta_t $, the process $(\eta_{t+s})_{0\le s<L}$ has the same distribution under $\mathbf{P}^{\rho,L}$ as $(\eta_s)_{0\le s<L}$ under ${\bf P}^{\eta_t }$. 
\\
We extend the notation to the random walker: we write $\mathbb P_{z}^{\rho,L}[\cdot] = \int \mathbf{P}^{\rho,L}(d\eta) P_z^{\eta}[\cdot]$ so that $\mathbb P^{\rho,\infty}_z=\mathbb P^\rho_z$ corresponds to the annealed law defined in~\eqref{eq:RW_ann}. We still abbreviate $\mathbb P^{\rho,L}=\mathbb P^{\rho,L}_0$.
\\
A quick inspection yields that 
\begin{equation}
\begin{split}
    &\text{for any environment satisfying~\ref{pe:markov}-\ref{pe:density} and any subset of the conditions~\ref{pe:densitychange}-\ref{pe:nacelle},} \\
    &\text{ its finite-range version satisfies the same for all $L\in \mathbb{N}$, and ${\mu}_\rho$ is still }
    \\
    &\text{an invariant measure of the finite-range version for the time evolution}.
    \end{split}
\end{equation}
\noindent
Finally, for all $\rho \in J, L\in \mathbb{N}$, define
\begin{equation} \label{eq: v(rho, L) def}
    v(\rho, L) = \dE\left[\frac{X^{\rho,L}_L}{L}\right] \,\left( =\dE\left[\frac{X^{\rho}_L}{L}\right] \right).
\end{equation}
Note that $X^{\rho,L}$ satisfies a law of large numbers with speed $v(\rho,L)$ as its trajectory can be sliced into i.i.d. blocks of length $L$,  
see~\cite[Lemma 3.2]{CKKR24} for details. 

\begin{proposition}
\label{prop: many to one coupling}
    For any environment satisfying (\ref{pe:markov})-(\ref{pe:density}), \ref{pe:drift}, and \ref{pe:couplings}, there exists $L_0\geq 1$ such that the following holds. For all integers $L > L_0$, $f(L) \in [(\log L)^{90}, L^{1/10}]$ and $n\geq 1$, for any density $\rho > f(L)^{-1/40}$, 
    there exists a coupling $\dQ_{n,L}$ of $(X^{(i)}_s)_{0 \leq s \leq n}$ for $i = 1,2$, 
    where $X^{(1)} \sim \dP^{\rho, n}$ and $X^{(2)} \sim \dP^{\rho + \varepsilon, L}$ with $\varepsilon \geq f(L)^{-1/40}$, such that 
    \begin{equation} \label{eq: many to one statement}
        \dQ_{n,L} \left( \min_{0 \leq s \leq n} \left(X^{(2)}_s - X^{(1)}_s \right) \leq - \left\lceil \frac{n}{L} -1 \right\rceil f(L) \right) 
        \leq \left( \left\lceil \frac{n}{L} \right\rceil -1 \right)^2  e^{-f(L)^{1/40}}.
    \end{equation}
    As a consequence, 
\begin{equation}\label{eq: many to one expectation}
    v(\rho,n)\leq v(\rho+\varepsilon,L)+2\frac{n- L\lfloor n/L\rfloor}{n}+\left\lceil \frac{n}{L} -1 \right\rceil \frac{f(L)}{n}  +2\left( \left\lceil \frac{n}{L} \right\rceil -1 \right)^2  e^{-f(L)^{1/40}}.
\end{equation}
    \noindent
    The same holds when $X^{(1)} \sim \dP^{\rho, L}$ and $X^{(2)} \sim \dP^{\rho + \varepsilon, n}$.
\end{proposition}

\begin{Rk} \label{Rk: Flitrations}
  We will denote by $(\mathcal{F}(t))_{t \geq 0}$ the natural filtration (in time) w.r.t. $X^{(1)}$ and $X^{(2)}$ and their environments, denoted $(\eta^{(1)}_t)_{t \geq 0}$ and $(\eta^{(2)}_t)_{t \geq 0}$ respectively. 
    We will apply several times the following principle: if an event is measurable with respect to one of these filtrations, we can make use of the Markov property (\ref{pe:markov}) to ensure that our coupling maintains the appropriate marginal distributions.
\end{Rk}

The proof is a generalisation of~\cite[Lemma 4.1]{CKKR24} which treated the case $n=2L$. We give a full account for completeness. 

\begin{proof}
    
    We only establish the proposition for $X^{(1)} \sim \dP^{\rho, n}$ and $X^{(2)} \sim \dP^{\rho + \varepsilon, L}$, as the case when $X^{(1)} \sim \dP^{\rho, L}$ and $X^{(2)} \sim \dP^{\rho + \varepsilon, n}$ follows in the same way. Note that~\eqref{eq: many to one statement} implies~\eqref{eq: many to one expectation} as, writing $\mathbb{E}_{ \dQ_{n,L} }$ for the expectation under $ \dQ_{n,L} $, 
\begin{equation} \label{eq: speed range adjustor}
\begin{split}
    v(\rho,n)-&v(\rho+\varepsilon,L)=\mathbb{E}_{ \dQ_{n,L} }\left[\frac{X^{(1)}_n}{n} \right]-\mathbb{E}_{ \dQ_{n,L} }\left[\frac{X^{(2)}_{L\lfloor \frac{n}{L}\rfloor}}{L\lfloor \frac{n}{L}\rfloor} \right]
\\
&\leq \mathbb{E}_{ \dQ_{n,L} }\left[\frac{X^{(1)}_{L\lfloor \frac{n}{L}\rfloor}+n-{L\lfloor \frac{n}{L}\rfloor}}{n} \right]-\mathbb{E}_{ \dQ_{n,L} }\left[\frac{X^{(2)}_{L\lfloor \frac{n}{L}\rfloor}}{L\lfloor \frac{n}{L}\rfloor} \right]
\\
&\leq\frac{n- L\lfloor n/L\rfloor}{n}+\frac{1}{n}\mathbb{E}_{ \dQ_{n,L} }\left[X^{(1)}_{L\lfloor \frac{n}{L}\rfloor}- X^{(2)}_{L\lfloor \frac{n}{L}\rfloor}\right]+\left(\frac{1}{n}-\frac{1}{L\lfloor \frac{n}{L}\rfloor}\right)\mathbb{E}_{ \dQ_{n,L} }\left[X^{(2)}_{L\lfloor \frac{n}{L}\rfloor}\right]
\\
&\leq \frac{n- L\lfloor n/L\rfloor}{n}+\left(\frac{n- L\lfloor \frac{n}{L}\rfloor}{n\times L\lfloor \frac{n}{L}\rfloor}\right)L\left\lfloor \frac{n}{L}\right\rfloor +\frac{1}{n}\mathbb{E}_{ \dQ_{n,L} }\left[X^{(1)}_{L\lfloor \frac{n}{L}\rfloor}- X^{(2)}_{L\lfloor \frac{n}{L}\rfloor}\right]
\\
&\leq 2\frac{n- L\lfloor n/L\rfloor}{n}+\frac{1}{n}\left(2L\left\lfloor \frac{n}{L}\right\rfloor \dQ_{n,L}\left( X^{(1)}_{L\lfloor \frac{n}{L}\rfloor}- X^{(2)}_{L\left\lfloor \frac{n}{L}\right\rfloor} > \left\lceil \frac{n}{L} -1 \right\rceil f(L)\right)  \right.
\\
&\quad\quad\quad\quad\,\,+\left. \left\lceil\frac{n}{L} -1 \right\rceil f(L)\right)
\end{split}
\end{equation}
and the conclusion follows from~\eqref{eq: many to one statement}. 
    
    We now focus on~\eqref{eq: many to one statement}. We will prove this statement for any $n = KL$, $K \in \dN$. 
    This implies the proposition in the general case, since we are considering the minimum over $0 \leq s \leq n$, so if $n \leq KL$, then 
    \begin{equation} \label{eq: many to one generalising KL to n pt 1}
        \dQ_{KL,L} \left( \min_{0 \leq s \leq KL} \left(X^{(2)}_s - X^{(1)}_s \right) \leq - \left\lceil \frac{KL}{L} -1 \right\rceil f(L) \right) 
        \leq \left( K-1 \right)^2 e^{-f(L)^{1/40}}
    \end{equation}
    implies
    \begin{equation} \label{eq: many to one generalising KL to n pt 2}
    \dQ_{n,L} \left( \min_{0 \leq s \leq n} \left(X^{(2)}_s - X^{(1)}_s \right) \leq - \left\lceil \frac{n}{L} -1 \right\rceil f(L) \right) 
        \leq  (K-1)^2 e^{-f(L)^{1/40}}.
    \end{equation}
    When $(K-1)L < n \leq KL$, $\left \lceil \frac{n}{L} \right \rceil = K$, and thus, the right hand side of \eqref{eq: many to one generalising KL to n pt 2} is equal to that of \eqref{eq: many to one statement}.
\\
\\
    The proof will proceed by inducting on $K$. 
\\
    The case $K = 1$ holds by (\ref{pe:monotonicity}), with the coupling $\dQ_{L,L}$ given by taking $\eta' = \eta^{(1)}$ and $\eta = \eta^{(2)}$. 
\\
    Now, assume that for some $K \geq 1$, $\dQ_{KL,L}$ is well defined (that is, with the correct marginals for the walkers and environments) and
    \begin{equation} \label{eq: many to one ind hyp}
        \dQ_{KL,L} \left( \min_{0 \leq s \leq KL} \left(X^{(2)}_s - X^{(1)}_s \right) \leq - (K-1) f(L) \right) 
        \leq (K-1)^2 e^{-f(L)^{1/40}}.
    \end{equation}

    Letting $n = (K+1)L$, we aim to  construct $\dQ_{(K+1)L,L}$ in a coherent fashion and to show that
    \begin{equation} \label{eq: many to one ind aim}
        \dQ_{n,L} \left( \min_{0 \leq s \leq n} \left(X^{(2)}_s - X^{(1)}_s \right) \leq - K f(L) \right) 
        \leq K^2 e^{-f(L)^{1/40}},
    \end{equation}
which would conclude the proof.

    \begin{figure}
        \centering
        \includegraphics[width=0.5\linewidth]{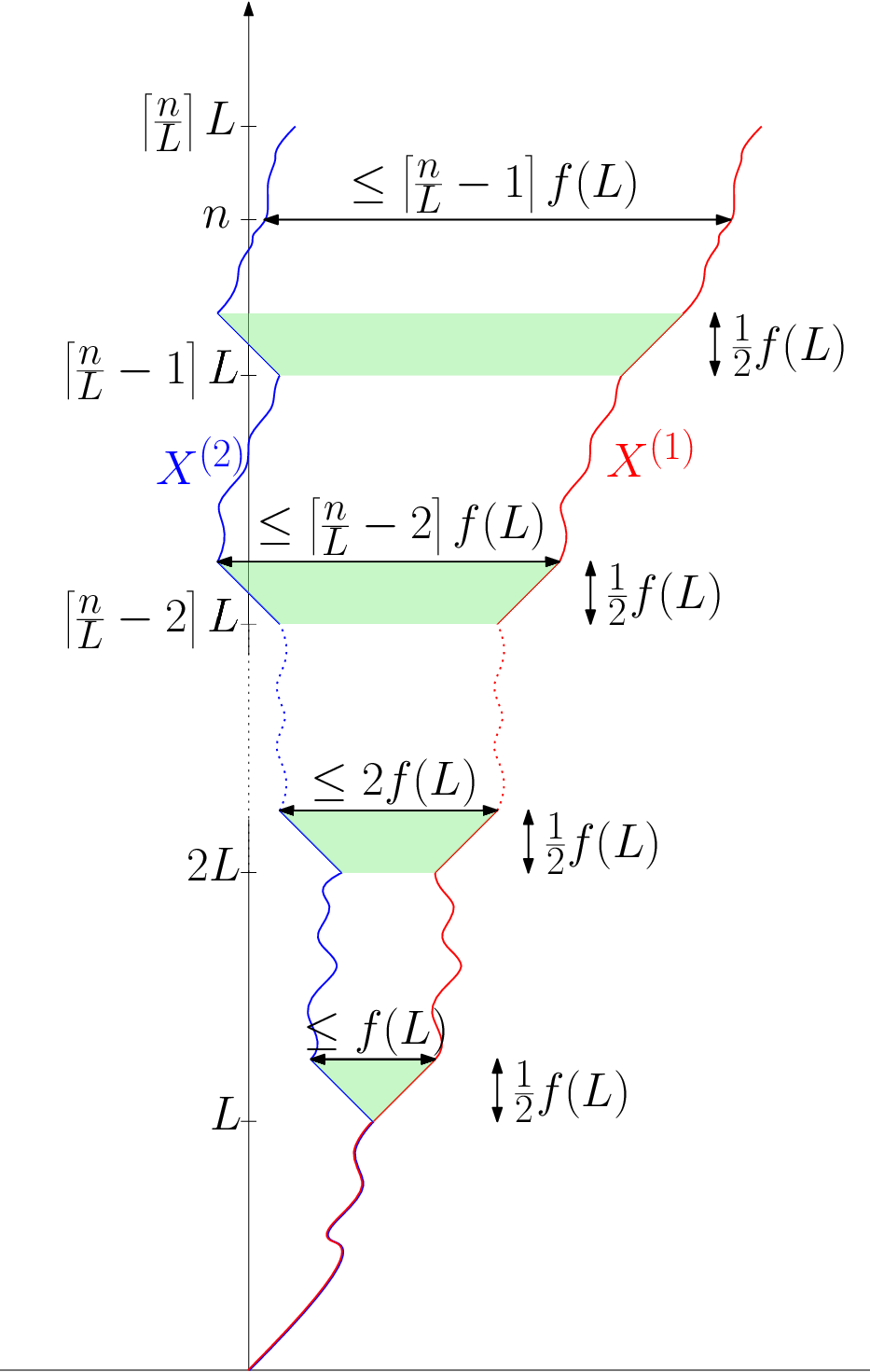}
        \caption{An illustration of $X^{(1)}$ (in red) and $X^{(2)}$ (in blue) under $\dQ_{n,L}$. We lose control of the exact coupling of the environments in the green shaded areas and attempt to recouple them by the exit time of those areas. During this time, $X^{(1)}$ could potentially drift ahead of $X^{(2)}$, incurring the loss term inside $\mathbb{Q}_{n,L}$ in~\eqref{eq: many to one statement}.}
        \label{fig:many to one proof}
    \end{figure}
We first construct the coupling $\dQ_{n, L}$ before bounding the error probabilities so as to satisfy~\eqref{eq: many to one statement}. In a nutshell, $\dQ_{n, L}$ coincides with $\dQ_{KL, L}$ on the time interval $[0, KL]$. Then at time $KL$, we apply the coupling  \ref{pe:couplings} between times $KL$ and $KL + f(L)/2$, on the event that the environments have the required empirical densities (event $G_1$ at  \eqref{eq: many to one G1 def}). If the coupling is successful (meaning that $\eta^{(2)}$ shifted by the worst-case discrepancy between $X^{(1)}$ and $X^{(2)}$ covers $\eta^{(1)}$, see $G_2$ at \eqref{eq: many to one G2 def}), we extend this domination until time $(K+1)L$ via \ref{pe:drift}. The three failure possibilities are: this coupling \ref{pe:drift}  (event $G_3^c$, see \eqref{eq: many to one G3 def}) or the previous coupling \ref{pe:couplings} being unsuccessful (hence $G_2^c$), or the conditions to apply \ref{pe:couplings} at time $KL$ not being met ($G_1^c$, see \eqref{eq: many to one G1 def}). In each case, we extend the environments according to the right marginals in any fashion (say, independently from each other), and from \ref{pe:drift}-\ref{pe:couplings} we can check (at \eqref{eq: many to one union bound proba}) that the probability of any of these happening is small enough. 
\\
\\
In detail, we first let $\dQ_{n, L}$ coincide with $\dQ_{KL, L}$ during $[0, KL]$. 
Consider the set of intervals
    \begin{equation}
        \cI = \left\{ I \subset [-5n, 5n], \lfloor \lfloor f(L)^{1/4}\rfloor/2\rfloor  \leq |I| \leq \lfloor f(L)^{1/4}\rfloor \right\}.
    \end{equation} 
    At time $KL$, we define the $\mathcal{F}(KL)$-measurable event
    \begin{equation} \label{eq: many to one G1 def}
        G_1 = \left\{ \forall I \in \cI, \eta^{(1)}_{KL}(I) \leq (\rho + \varepsilon/4) |I| \text{ and } \eta^{(2)}_{KL}(I) \geq (\rho + 3\varepsilon/4) |I| \right\}.
    \end{equation}
  For simplicity, we will write $x_1=X^{(1)}_{KL}$ and $x_2=X^{(2)}_{KL}$ in the sequel. 
  On any realisation of $(\eta^{(1)}_t,\eta^{(2)}_t, X^{(1)}_t, X^{(2)}_t)_{0\leq t \leq KL}$ such that the event $G_1^c$ holds, extend the coupling until time $n$ by letting $\eta^{(1)}$ (resp. $\eta^{(2)}$) evolve according to $\mathbf{P}^{\eta^{(1)}_{KL}}$ (resp. $\mathbf{P}^{\eta^{(2)}_{KL}}$), independently of each other, and given the realisations of these environments, let again $X^{(1)}\sim \mathbb P^{\eta^{(1)}}_{x_1}$ and $X^{(2)}\sim \mathbb P^{\eta^{(2)}}_{x_2}$ independently until time $n$. As noted in Remark~\ref{Rk: Flitrations}, the $\mathcal{F}(KL)$-measurability of $G_1$ and~\ref{pe:markov} ensure that the environments and walks retain the desired marginals over the full interval $[0,n]$. 
  \\
  In the same logic, on any realisation of $(\eta^{(1)}_t,\eta^{(2)}_t, X^{(1)}_t, X^{(2)}_t)_{0\leq t \leq KL}$ such that $G_1$ holds, we couple $\eta^{(1)}, \eta^{(2)}$ as in \ref{pe:couplings} with $H = 7L, t = \lfloor f(L)/2\rfloor $, $\eta'=\eta^{(1)}_{KL+\cdot}(x_1+\cdot)$, $\eta=\eta^{(2)}_{KL+\cdot}(x_2-2t+\cdot )$ and keeping the same values for $\rho, \varepsilon$. Note that the conditions of \ref{pe:couplings} are met on $G_1$ as $\vert x_1\pm H\vert,\vert x_2-2t\pm H\vert \leq H+ KL+f(L)\leq 7L+n \leq 9n/2 <5n$, since $n=(K+1)L\geq 2L$ (and $f(L)\leq \max((\log L)^{90}, L^{1/10})\leq L$ if $L_0$ is chosen large enough).  
   Define the event
    \begin{equation} \label{eq: many to one G2 def}
        G_2 = \left\{ \forall x \in [-3L, 3L], \eta^{(1)}_{KL+\lfloor f(L)/2\rfloor}(x + x_1) \leq \eta^{(2)}_{KL+\lfloor f(L)/2\rfloor}(x-2\lfloor f(L)/2\rfloor+x_2) \right\},
    \end{equation} 
    which is what we deem a successful recoupling of the environments. 
    \\
On the event $G_2^c$, like for $G_1^c$, we extend the coupling on the time interval $[KL+\lfloor f(L)/2\rfloor,n]$ by letting both environments and walks evolve independently according to their marginal distributions, and note again that this retains the desired marginals by \ref{pe:markov} and the $\mathcal{F}(KL+\lfloor f(L)/2\rfloor)$-measurability of $G_2$.
\\
On any realisation of the environments and walks such that $G_2$ holds, we can apply \ref{pe:drift} with $\eta =\eta^{(2)}_{KL+\lfloor f(L)/2\rfloor+\cdot}(x_2-2\lfloor f(L)/2\rfloor+\cdot)$, $\eta'=\eta^{(1)}_{KL+\lfloor f(L)/2\rfloor+\cdot}(x_1+\cdot)$, $H=3L$, $t=L-\lfloor f(L)/2\rfloor$ and $k=1$. Once again due to the $\mathcal{F}(KL+\lfloor f(L)/2\rfloor)$-measurability of $G_2$ and~\ref{pe:markov}, we now have that 
\begin{equation}\label{eq: many to one marginal laws envt}
\begin{split}
     &\text{$\eta^{(1)}, \eta^{(2)}$ are distributed as $\mathbf{P}^{\rho, L},\mathbf{P}^{\rho+\varepsilon, n}$ respectively }
     \\
     &\text{over the full time interval $[0,n]$, for $i=1,2$.}
\end{split}
\end{equation}

We define a last event (corresponding to the coupling \ref{pe:drift} being successful)  
\begin{equation} \label{eq: many to one G3 def}
\begin{split}
   &G_3 = 
   \\
   &\left \{ \forall s \in [KL + \lfloor f(L)/2\rfloor, n], \eta^{(1)}_s|_{[-L + x_1, L + x_1]} \preccurlyeq \eta^{(2)}_s|_{[-L - 2\lfloor f(L)/2\rfloor + x_2, L - 2\lfloor f(L)/2\rfloor + x_2]}
            \right \}.    
\end{split}
    \end{equation}
On the event $G_3^c$, extend the coupling during $[KL + \lfloor f(L)/2\rfloor, n]$ by letting $X^{(i)}$ evolve according to $\mathbb{P}^{\eta^{(i)}}_{X^{(i)}_{KL + \lfloor f(L)/2\rfloor}}$, independently for $i=1,2$. Finally, on any realisation of the environments $\eta^{(1)},\eta^{(2)}$ such that $G_3$ holds, we let the two walks $X^{(1)},X^{(2)}$ evolve by using a same collection of i.i.d. uniform variables $U_n$'s over $[0,1]$ (independent of everything else) as per \eqref{e:def-RWDRE}. 
\\
Remark that by our construction, on every time interval $[a,b]$ and regardless of which branch of the scenario we are on, for any given realisation of $\eta^{(i)}$ we have $X^{(i)}\sim \mathbb{P}^{\eta^{(i)}}_{X^{(i)}_a}$, $i=1,2$. Combining this with \eqref{eq: many to one marginal laws envt} yields that both $X^{(1)} \sim \dP^{\rho, n}$ and $X^{(2)} \sim \dP^{\rho + \varepsilon, L}$ under $\dQ_{n,L}$, as desired. 
\\
\\
We now turn to the proof of~\eqref{eq: many to one ind aim}. \\
    Define $E_{M} = \left\{ \min_{0 \leq s \leq ML} \left( X^{(2)}_s - X^{(1)}_s   \right) \leq -(M-1)f(L) \right\}$, for $M\geq 1$. Note that 
    \begin{equation}
        \begin{split}
          E_{K+1}
          \subseteq 
          E_{K} \bigcup
          \left(\left\{ \min_{KL \leq s \leq n} \left( X^{(2)}_s - X^{(1)}_s   \right) \leq -K f(L) \right\} \bigcap E_{K}^c \right).
        \end{split}
    \end{equation}
Thus by a union bound and the induction hypothesis applied to $E_K$ (noticing that $\dQ_{n,L}(E_K) = \dQ_{KL,L}(E_K) $ by construction as $E_K$ is $\mathcal{F}(KL)$-measurable), it suffices to show that
  \begin{equation} \label{eq: many to one ind step requirement}
        \begin{split}
            \dQ_{n,L} \bigg( \left( \min_{KL \leq s \leq (K+1)L} \left( X^{(2)}_s - X^{(1)}_s \right) \leq -K f(L) \right) \bigcap E_K^c \bigg) 
            \leq  \left(2K -1 \right) e^{-f(L)^{1/40}}, 
        \end{split}
    \end{equation}
which will follow once we establish 
      \begin{equation}\label{eq: many to one event inclusion}
        \left\{ \min_{KL \leq s \leq n} \left( X^{(2)}_s - X^{(1)}_s          \right) \leq -K f(L) \right\} \bigcap E_{K}^c 
        \subseteq  G_1^c \bigcup G_2^c \bigcup G_3^c
    \end{equation}
and
\begin{equation}\label{eq: many to one three Gs}
     \dQ_{n,L}(G_1^c \cup G_2^c \cup G_3^c) \leq \left(2K -1 \right) e^{-f(L)^{1/40}}
\end{equation}
for $L$ large enough, uniformly in $K$. 
\\
\textbf{Proof of~\eqref{eq: many to one event inclusion}:}
We actually show that on $E_{K}^c \cap  G_1\cap G_2\cap G_3$, we have 
\begin{equation}
    \min_{KL \leq s \leq n} \left( X^{(2)}_s - X^{(1)}_s  \right) > -K f(L),        
\end{equation}
which suffices for our purpose. Note that on $E_K^c$, we have $x_1< x_2+(K-1) f( L )$, and since the walkers move by one space unit at every time step, we get
\begin{equation} \label{eq: many to one walks post recoupling}
X^{(1)}_{KL+\lfloor f(L)/2\rfloor }  < X^{(2)}_{KL+\lfloor f(L)/2\rfloor }   + (K-1){f(L)} \color{black} +2\lfloor f(L)/2\rfloor \leq  X^{(2)}_{KL+\lfloor f(L)/2\rfloor }   + K
{f}(L).
\end{equation}

From here, using $G_2$ and $G_3$, we have that
\begin{equation}\label{eq: eta 1 smaller eta 2}
    \eta^{(1)}_s(x + x_1) \leq \eta^{(2)}_s\left( x - 2\left\lfloor f(L)/2 \right\rfloor + x_2 \right)
\end{equation}
for all $x\in [-L, L], s \in [KL + \lfloor f(L)/2 \rfloor, n]$. 

Define two auxiliary walks $Y^{(1)} \sim \dP^{\eta^{(1)}}, Y^{(2)} \sim \dP^{\eta^{(2)}}$ starting at time $KL+\lfloor f(L)/2 \rfloor$ with initial positions $Y^{(1)}_{KL+\lfloor f(L)/2 \rfloor} = x_1 + \lfloor f(L)/2 \rfloor$ and $Y^{(2)}_{KL+\lfloor f(L)/2 \rfloor} = x_2 - \lfloor f(L)/2 \rfloor$. We can apply Lemma \ref{lem:monotone} with $\eta(\cdot) = \eta^{(1)}(\cdot + x_1)$, $\tilde \eta (\cdot) = \eta^{(2)}(\cdot - 2\left\lfloor f(L)/2 \right\rfloor + x_2)$ 
and $w' = w = (\lfloor f(L)/2 \rfloor, KL + \lfloor f(L)/2 \rfloor)$, owing to~\eqref{eq: eta 1 smaller eta 2}. We obtain that
\begin{equation} \label{eq: many to one worst case monotone comp}
    Y^{(1)}_s \leq Y^{(2)}_s + x_1 - x_2 + 2 \lfloor f(L)/2 \rfloor
    \leq Y^{(2)}_s + Kf(L)
\end{equation}
for $s \in [KL + \lfloor f(L)/2 \rfloor, n]$.

We can then compare $X^{(1)}$ to $Y^{(1)}$ using Lemma \ref{lem:monotone} with $\eta(\cdot) = \tilde\eta(\cdot) = \eta^{(1)}(\cdot)$ and $w' = (X^{(1)}_{KL + \lfloor f(L)/2 \rfloor}, KL + \lfloor f(L)/2 \rfloor), w =( x_1 + \lfloor f(L)/2 \rfloor, KL + \lfloor f(L)/2 \rfloor)$. Note that as $X^{(1)}_{KL} = x_1$, we have that $\pi_1(w') \leq \pi_1(w)$. We obtain that
\begin{equation} \label{eq: many to one lower density walk to worst case monotone comp}
    X^{(1)}_s \leq Y^{(1)}_s
\end{equation}
for $s \in [KL + \lfloor f(L)/2 \rfloor, n]$.

Similarly, we can compare $X^{(2)}$ to $Y^{(2)}$ using again Lemma \ref{lem:monotone} and we 
obtain that
\begin{equation} \label{eq: many to one higher density walk to worst case monotone comp}
    X^{(2)}_s \geq Y^{(2)}_s
\end{equation}
for $s \in [KL + \lfloor f(L)/2 \rfloor, n]$.

By combining \eqref{eq: many to one worst case monotone comp}, \eqref{eq: many to one lower density walk to worst case monotone comp}, and \eqref{eq: many to one higher density walk to worst case monotone comp}, we have that
\begin{equation}
    X^{(1)}_s
    \leq X^{(2)}_s + Kf(L)
\end{equation}
for $s \in [KL + \lfloor f(L)/2 \rfloor, n]$.

\noindent
\textbf{Proof of~\eqref{eq: many to one three Gs}:}
First, we obtain by (\ref{pe:density}) with $\ell$ ranging in $[\lfloor \lfloor f(L)^{1/4}\rfloor/2\rfloor ,\lfloor f(L)^{1/4}\rfloor ]$ and $\varepsilon/4$ instead of $\varepsilon$, and a union bound on $\mathcal{I}$:
    \begin{equation} \label{eq: many to one G1 prob}
        \dQ_{n,L} (G_1^c) \leq 20 n f(L)^{1/4}\exp \left( -\Cr{densitydev}\left(\frac{\varepsilon}{4}\right) ^2 \lfloor \lfloor f(L)^{1/4}\rfloor/2\rfloor  \right) 
    \end{equation}
Second, our application of \ref{pe:couplings} to any realisation $R$ of $\eta^{(1)}, \eta^{(2)}, X^{(1)},$ and $X^{(2)}$ up to time $KL$ such that $G_1$ holds, above~\eqref{eq: many to one G2 def}, yields 
    \begin{equation} \label{eq: many to one G2 prob}
        \dQ_{n,L} \left( G_2^c | R \right) \leq 7\Cr{SEPcoupling2} f(L) L \exp \left( - (\Cr{SEPcoupling2} (1 + \nu^{-1}))^{-1}\left(\frac{\varepsilon}{4}\right)^2f(L)^{1/4} \right).
    \end{equation}
Third, when we apply \ref{pe:drift} at time $KL+\lfloor f(L)/2\rfloor $ on any realisation $R$ of $\eta^{(1)}, \eta^{(2)}, X^{(1)},$ and $X^{(2)}$ such that $G_1$ and $G_2$ hold, above~\eqref{eq: many to one marginal laws envt}, we obtain 
    \begin{equation} \label{eq: many to one G3 prob}
        \dQ_{n,L}(G_3^c|R) \leq 20 \exp (-(L - f(L))/4).
    \end{equation} 
Putting together \eqref{eq: many to one G1 prob}, \eqref{eq: many to one G2 prob} and \eqref{eq: many to one G3 prob} and substituting $ \varepsilon=f(L)^{-1/40}$ yields 
    \begin{equation}\label{eq: many to one union bound proba}
        \begin{split}
            \dQ_{n,L}&(G_1^c \cup G_2^c \cup G_3^c) 
            \leq 47\Cr{SEPcoupling2}\frac{n}{L}L^3\exp(-cf(L)^{1/5})\leq (2K-1)\exp({-f(L)^{1/40}})
        \end{split}
    \end{equation}
    for some constant $c=c(c_1,\Cr{SEPcoupling2})>0$, and for $L_0$ large enough in a way that only depends on $c_1$ and $\Cr{SEPcoupling2}$ (essentially, $L$ and $f(L)$ must be large enough so that the terms in each exponential dominate $f(L)^{1/40}$ up to the ad hoc prefactors). This concludes the proof. 
\end{proof}

\section{Proof of Theorem \ref{thm:moregeneral}}\label{sec:speed}

In this section, we prove \Cref{thm:moregeneral}. We will first find a candidate speed for all but countably many $\rho \in J$ in \Cref{sec:speedcandidate}. Then, we will prove in Section~\ref{subsec:speedproof} that this candidate is indeed the speed of the random walk with density $\rho$.

\subsection{Speed candidate}\label{sec:speedcandidate}

We will use the following proposition to show that the limit in $L$ of $    v(\rho, L) $ must exist for all but countably many $\rho \in J$. 
We will refer to this limit in $L$ as $v(\rho)$, and thus define (where it exists)
\begin{equation} \label{def: v(rho)}
    v(\rho) = \lim_{L \rightarrow \infty} v(\rho, L).
\end{equation}

\begin{proposition} \label{prop: liminf limsup}
    Let $\eta$ be an environment satisfying~\ref{pe:markov}-\ref{pe:density},  \ref{pe:drift} and \ref{pe:couplings}.  Then, for all $\rho \in J, \varepsilon > 0$, 
    \begin{equation} \label{eq: limit inequality in rho}
        \limsup_{L \rightarrow \infty} v(\rho, L) \leq \liminf_{L \rightarrow \infty} v(\rho + \varepsilon, L).
    \end{equation}
As a consequence, if $\widebar{J}:=\{\rho \in J,     {v}(\rho)\text{ is well defined} \}$, then $J\setminus \widebar{J}$ is countable. Moreover, ${v}$ is non-decreasing over $\widebar{J}$. 
\end{proposition} 
In a nutshell, the proof relies on approximating arbitrarily the limsup and the liminf by $v(\rho,k)$ and $v(\rho+\varepsilon,k')$ for large enough $k$ and $k'$. Note that it may not be possible to compare them directly via Proposition~\ref{prop: many to one coupling}, as we have no quantitative estimate on the relation between $k$ and $k'$, so that the error term $k/k'$ (or $k'/k$ if $L'>L$) could be arbitrarily large. Instead, to bridge the gap from $k$ to $k'$, we define a sequence of couplings as in Proposition~\ref{prop: many to one coupling} with $n\leq 2L$ each time, gradually sprinkling the extra-density of particles up to an amount not exceeding $\varepsilon$. 

\begin{proof}
Let $\rho \in J$ and $\varepsilon > 0$. Fix $\delta \in (0,1)$. Then there exist two large enough integers $k,k'$ with $k'\geq k/\delta$ such that 
\begin{equation}\label{eq: v rho delta inequality}
v(\rho, k) \geq \limsup_{L \rightarrow \infty} v(\rho, L)\,-\delta \text{ and }v(\rho+\varepsilon, k') \leq \liminf_{L \rightarrow \infty} v(\rho+\varepsilon, L)\,+\delta.
\end{equation}
Let $j\geq 0$ be the unique integer such that $2^jk/\delta\in [k',2k')$. For $0\leq i\leq j$, set $L_i:=2^ik$, and write $L_{j+1}:=k'$. Let $f:x \mapsto x^{1/10}$ over $\mathbb{R}_+$, and define for $1\leq i\leq j$:
\begin{equation}
 \varepsilon_i=f(L_i)^{-1/40},\,\rho_i=\rho+\sum_{u=1}^i\varepsilon_u.
\end{equation}
Let also $\rho_0:=\rho$. 
By Proposition~\ref{prop: many to one coupling} with $n=L_{i+1}=2L_i,L=L_i, \rho=\rho_i $ and $\varepsilon=\varepsilon_i$, we have for all $0\leq i \leq j-1$:
\begin{equation}\label{eq: v chain inequality}
v(\rho_i,L_i) \leq v(\rho_{i+1}, L_{i+1})+\frac{f(L_i)}{2L_{i}}+2e^{-f(L_i)^{1/40}}.    
\end{equation}
We then need to apply a last time the coupling of Proposition~\ref{prop: many to one coupling} to connect the scales $2^jk$ and $k'$, with $n=L_{j+1}$, $L=L_j$, $\rho=\rho_j$, and $\varepsilon$ replaced by $\varepsilon - \sum_{u=1}^j\varepsilon_u$. Note that this is possible since 
\begin{equation}
\varepsilon - \sum_{u=1}^j\varepsilon_u = \varepsilon -\sum_{u=1}^j (2^uk)^{-1/400}\geq\varepsilon- k^{-1/400}\sum_{u\geq 1}2^{-u/400} \geq (2^{j+1}k)^{-1/400} = f(L_{j+1})^{-1/40},    
\end{equation}
provided $k$ is chosen large engouh w.r.t. $\varepsilon$. Using that $\delta L_{j+1}\leq L_j < 2\delta L_{j+1}$ by definition of $j$, this yields
\begin{equation}\label{eq: v chain last inequality}
\begin{split}
v(\rho_j, L_j)&\leq v(\rho_{j+1}, L_{j+1}) +  2\frac{L_j}{L_{j+1}}+\left\lceil \frac{L_{j+1}}{L_j} -1 \right\rceil \frac{f(L_j)}{L_{j+1}}  +2\left( \left\lceil \frac{L_{j+1}}{L_j} \right\rceil -1 \right)^2  e^{-f(L_j)^{1/40}} 
\\
&\leq v(\rho_{j+1}, L_{j+1}) +4\delta + \delta^{-1}\frac{(2^jk)^{1/10}}{2^{j-1}k}+2\delta^{-2}\exp(-{(2^jk)}^{1/400})
\\
&\leq v(\rho_{j+1}, L_{j+1}) +4\delta + 2\delta^{-1}k^{-9/10}+2\delta^{-2}\exp(-{k}^{1/400})
\\
&\leq v(\rho_{j+1}, L_{j+1}) +5\delta 
\end{split} 
\end{equation}
if we choose $k$ large enough w.r.t.~$\delta$. 
Putting~\eqref{eq: v chain inequality} and~\eqref{eq: v chain last inequality} together, we obtain 
\begin{equation}
\begin{split}
 v(\rho,k)&\leq v(\rho+\varepsilon,k')+5\delta +\sum_{i=0}^{j-1}\left(L_i^{-9/10}+2e^{-L_i^{1/400}} \right)
 \\
 &\leq v(\rho+\varepsilon,k')+5\delta +\sum_{i=0}^{j-1}\left(k^{-9/10}2^{-9i/10}+2e^{-(1+ i (\log 2)/400)k^{1/400}} \right)
  \\
 &\leq v(\rho+\varepsilon,k')+6\delta
\end{split}
\end{equation}
using $2^{x}=e^{x\log 2}\geq 1+x\log 2$ in the second inequality, and if we choose again $k$ large enough w.r.t.~$\delta$. 
As $\delta >0$ can be taken arbitrarily small,~\eqref{eq: limit inequality in rho} follows from this and~\eqref{eq: v rho delta inequality}. 
\\


\noindent
Finally, consider the family of intervals $(I_\rho)_{\rho \in J}$, where 
\begin{equation}
    I_\rho = \left( \liminf_{L \rightarrow \infty} v(\rho, L), \limsup_{L \rightarrow \infty} v(\rho, L) \right).
\end{equation}
Note that $(I_\rho)_{\rho \in J}$ are disjoint by \eqref{eq: limit inequality in rho}. Moreover, since $I_\rho \subset [-1,1]$ for all $\rho \in J$, we have that
\begin{equation}
    \sum_{\rho \in J} |I_\rho| < 2.
\end{equation}
Therefore, the set $\left\{ \rho \in J : |I_\rho| \neq 0 \right\}$ must be countable. 
This set is precisely those $\rho$ such that ${v}(\rho) = \lim_{L \rightarrow \infty} v(\rho, L)$ does not exist, so $J \setminus \widebar{J}$ is countable.

\end{proof}

\subsection{Deviation Lemma and proof of Theorem \ref{thm:moregeneral}}\label{subsec:speedproof}
For all $\rho \in J$, denote by ${v}^-(\rho)$ and ${v}^+(\rho)$ the limits from below and above, respectively, of ${v}(\rho')$ as $\rho' \rightarrow \rho$, i.e. 
\begin{equation} \label{def:vbar+-}
    \begin{split}
        {v}^+ (\rho) &= \lim_{(\rho_n') \searrow \rho, \rho_n'\in \widebar{J}} {v}(\rho'_n) \\
        {v}^- (\rho) &= \lim_{(\rho_n') \nearrow \rho, \rho_n' \in \widebar{J}} {v}(\rho'_n).
    \end{split}
\end{equation}
Note that since $J\setminus \widebar{J}$ is at most countable, such sequences $(\rho_n')_{n\geq 1}$ exist, and that by monotonicity of ${v}$, these limits are well-defined, in the sense that they do not depend on the choice of the sequences $(\rho_n')$.

In this subsection, we prove Theorem \ref{thm:moregeneral} by using an upper bound on the probability of the following deviation event in Lemma~\ref{lem: generalised E_n^delta} (then coupled with the Borel-Cantelli Lemma), and the fact that ${v}^+(\rho) = {v}^{-}(\rho) =v(\rho)$ for all $\rho \in J$ but at most countably many (a direct consequence of Proposition~\ref{prop: liminf limsup}).

For all $\rho,\delta > 0$, define the event 
    \begin{equation} \label{eq: def generalised E_n^delta}
           E_n^\delta = \left\{\frac{X_n}{n} \notin \left( {v}^-(\rho) - \delta, {v}^+(\rho) + \delta \right)\right\}.
    \end{equation}

\begin{lemma}[Deviation Lemma] \label{lem: generalised E_n^delta}
    Let $\rho \in J$ and $\delta >0$. Then there exists $k_0=k_0(\rho, \delta)\in \mathbb{N}$ so that for all 
   integers $n\geq L> k_0$ with $n\geq 12L/\delta$ and all  $f (L)\in [(\log L)^{90}, L^{1/10}]$, 
    \begin{equation}\label{eq: generalised Endeltabound}
        \dP^\rho(E_n^\delta )\leq 2\left( \left\lceil \frac{n}{L} \right\rceil -1 \right)^2 \exp \left(-f(L)^{1/40} \right) + 2\exp \left(\frac{ - \delta^2 n}{2^9 L } \right).
    \end{equation} 
\end{lemma}

With this lemma, we will proceed with the proof of Theorem \ref{thm:moregeneral}.
\begin{proof}[Proof of Theorem~\ref{thm:moregeneral}]
 
Let $\rho$ be fixed, such that ${v}$ is defined and continuous at $\rho$ (on $\widebar{J}$). This potentially eliminates at most countably many options for $\rho$, since ${v}$ is non-decreasing by Proposition \ref{prop: liminf limsup}, hence we can afford this continuity restriction. 
We thus have  ${v}^-(\rho) = {v}(\rho) = {v}^+(\rho)$. Now, for any fixed $\delta >0$, using Lemma \ref{lem: generalised E_n^delta} with $L = \lfloor \sqrt{n} \rfloor$ (and, say, $f(x)  = x^{1/10}$), we easily derive that $\dP^\rho (E_n^\delta)$ decays faster than any polynomial in $n$ so that

\begin{equation}
  \sum_{n = 1}^\infty \dP^\rho (E_n^\delta) < \infty.
\end{equation}
Thus by Borel-Cantelli's Lemma, we have that almost surely,  $E_n^\delta$ only happens for finitely many $n$. 	This holds almost surely jointly for all $\delta \in \{\frac{1}{m}, m \in \dN \}$, which implies that $\frac{X_n}{n} \rightarrow v(\rho)$ as $n \rightarrow \infty$.

As for the monotonicity of $v$, it follows from Propositions 3.3 (restricted to $\rho \in \widebar{J}$) and 3.4 in~\cite{CKKR24}, which apply to any setup satisfying  the properties~\ref{pe:markov}-\ref{pe:density} and the conditions~\ref{pe:densitychange}-\ref{pe:nacelle}. 
 
 Note that the changes made to \ref{pe:nacelle} between~\cite{CKKR24} and ours do not impact the results. In particular, the changes made impacts equations (5.78), (5.81), and (5.83) in \cite{CKKR24}, but as $\delta'$ is exponential in $\ell$ and what $\delta'$ is being compared to in all of these equations is exponential in a superlinear function of $\ell$, all of these equations still hold. 
   
\end{proof}
\noindent
All that remains now is to prove Lemma \ref{lem: generalised E_n^delta}. 
\begin{proof}[Proof of Lemma \ref{lem: generalised E_n^delta}] Fix $\rho \in J$, $\delta >0$. 
    
    We will split $E_n^\delta$ into two events such that $E_n^\delta = E_n^{\delta, +} \cup E_n^{\delta, -}$, where
    \begin{equation}
        \begin{split}
            E_n^{\delta, +} &= \left\{ \frac{X_n}{n} \geq {v}^+(\rho) + \delta \right\} \\
            E_n^{\delta, -} &= \left\{ \frac{X_n}{n} \leq {v}^-(\rho) - \delta \right\}
        \end{split}
    \end{equation}
    We will only bound the probability of $E_n^{\delta, +}$, as the case of $E_n^{\delta, -}$ is symmetric.  Remark first that since ${v}^+(\rho)$ is the limit of ${v}(\rho')$ as $\rho' \downarrow \rho$ on $\widebar{J}$, there must exist some $\varepsilon_0 > 0$ such that ${v}(\rho + \varepsilon_0)$ exists  (as $J\setminus \widebar{J}$ is at most countable) and
\begin{equation} \label{eq: generalised Endelta+bound3 part 1}
    {v}(\rho + \varepsilon_0) < {v}^+(\rho) + \delta/6.
\end{equation}

Now, recall from \eqref{def: v(rho)} that 
\begin{equation}
    {v}(\rho + \varepsilon_0) = \lim_{L \rightarrow \infty} v(\rho + \varepsilon_0, L)
\end{equation}
so for $k_0$ (and therefore $L$)  large enough, 
\begin{equation} \label{eq: generalised Endelta+bound3 part 2}
    \dE^{\rho + \varepsilon_0,L} \left[ \frac{X_L}{L} \right] = v(\rho + \varepsilon_0, L) < {v}(\rho + \varepsilon_0) + \delta/6 < {v}^+(\rho) + \delta/3.
\end{equation}
    For $L\geq 1$, write $\varepsilon_L = f(L)^{-1/40}$ where $f$ is as in the statement. 
Note that by \ref{pe:monotonicity} and the fact that $\varepsilon_L < \varepsilon_0$ for large enough $L$, we have that 
\begin{equation} \label{eq: generalised Endelta+bound3 part 3}
    \dE^{\rho + \varepsilon_0,L} \left[ \frac{X_{L}}{L} \right] \geq \dE^{\rho + \varepsilon_L,L} \left[ \frac{X_{L}}{L} \right].
\end{equation} 

Additionally, note as in~\eqref{eq: speed range adjustor} (with one less term as the densities are the same here) that 
\begin{equation}
    \dE^{\rho + \varepsilon_L, L} \left[ \frac{X_n}{n} \right]\leq \dE^{\rho + \varepsilon_L, L} \left[ \frac{X_L}{L} \right] + 2\frac{n- L\lfloor \frac{n}{L}\rfloor }{n} \leq \dE^{\rho + \varepsilon_L, L} \left[ \frac{X_L}{L} \right] + \delta/6,
\end{equation}
using $n- L\lfloor \frac{n}{L}\rfloor \leq L-1 \leq n\delta/12$. 
Along with \eqref{eq: generalised Endelta+bound3 part 2} and \eqref{eq: generalised Endelta+bound3 part 3}, this yields that
\begin{equation}
    \dE^{\rho + \varepsilon_L, L} \left[ \frac{X_n}{n} \right] \leq {v}^+(\rho) + \delta/2.
\end{equation}

Hence~\eqref{eq: generalised Endeltabound} will follow once we show that, for sufficiently large $n$ and $L$, 
\begin{equation}\label{eq: generalised Endelta+bound2}
\dP^{\rho + \varepsilon_L, L}\left( \frac{X^{}_n}{n} \geq \dE^{\rho + \varepsilon_L, L} \left[ \frac{X^{}_n}{n} \right] + \delta/4 \right) < \exp \left(\frac{ - \delta^2 n}{2^9L } \right),
\end{equation} 
and there exists a coupling $\mathbb{Q}$ of $X^{(1)}\sim \mathbb{P}^{\rho, n} $ (or $X^{(1)}\sim \mathbb{P}^{\rho} $, which is effectively the same for the below) and $X^{(2)}\sim \mathbb{P}^{\rho+\varepsilon_L, L} $ such that
\begin{equation}\label{eq: generalised Endelta+bound1}
\mathbb{Q} \left( \frac{X^{(1)}_n}{n} \geq \frac{X^{(2)}_n}{n} + \delta/4 \right) 
    \leq \left( \left\lceil \frac{n}{L} \right\rceil -1 \right)^2 e^{-f(L)^{1/40}}.
\end{equation}

\textbf{Proof of (\ref{eq: generalised Endelta+bound2}) }
We use Azuma's inequality.
Let $Y_k = X_{kL} - X_{(k-1)L} - \Eannealed^{\rho + \varepsilon_L, L}[X_L]$ for $1 \leq k \leq \lfloor n/L \rfloor$, where $X\sim \mathbb{P}^{\rho+\varepsilon_L,L}$. 
	Then $(Y_k)_{1 \leq k \leq \lfloor n/L\rfloor }$ are i.i.d. and $|Y_k|$ is bounded by $2L$. 
	Now, define also $Y_{ \lfloor n/L \rfloor +1} = X_{n} - X_{L\lfloor n/L \rfloor} - \Eannealed^{\rho + \varepsilon_L, L}[X_{n - L \lfloor n/L \rfloor}]$.
	$Y_{\lceil n/L \rceil}$ is also independent from all other $Y_k$ and is also bounded by $2L$. 
	Note that for all $1 \leq k \leq  \lfloor n/L \rfloor +1$, $\Eannealed[Y_k] = 0$. 
	
	Then by Azuma's inequality,
	\begin{equation}
		\begin{split}
			\Pannealed^{\rho + \varepsilon_L, L} \left( \frac{X_n}{n}  - \Eannealed^{\rho + \varepsilon_L, L}  \left[ \frac{X_{n}}{n} \right] 
				> \frac{\delta}{4} \right)
				&= \Pannealed^{\rho + \varepsilon_L, L}  \left(\sum_{k=1}^{\lfloor n/L \rfloor +1} Y_{k} > n\delta/4 \right) \\
				&\leq \exp \left(\frac{ - (\delta/4)^2 n^2}{2 \sum_{k = 1}^{ \lfloor n/L \rfloor +1} (2L)^2} \right) \\
				&\leq \exp \left(\frac{ - \delta^2 n^2}{2^6 (2L)^2(\frac{n}{L}+1)} \right) \\
				&\leq \exp \left(\frac{ - \delta^2 n^2}{2^9 nL } \right) \\
				&\leq \exp \left(\frac{ - \delta^2 n}{2^9 L } \right).
		\end{split}
	\end{equation}

\textbf{Proof of (\ref{eq: generalised Endelta+bound1}) }
This will hold once we ensure that we can apply Proposition \ref{prop: many to one coupling} here (with $\varepsilon=\varepsilon_L$, all other variables being the same). To this end, note that 
\begin{equation}
    \left\{ \frac{X^{(1)}_n}{n} \geq \frac{X^{(2)}_n}{n} + \delta/4 \right\} = \left\{ X^{(2)}_n - X^{(1)}_n \leq  - \delta n/4 \right\}
\end{equation}
and
\begin{equation}
     - \left\lceil \frac{n}{L} - 1 \right\rceil f(L) \geq    - \left\lceil \frac{n}{L} - 1 \right\rceil L^{1/10}\geq - \delta n / 4
\end{equation}
for any given $\delta$, choosing $L$ sufficiently large.

\end{proof}

Lemma~\ref{lem: generalised E_n^delta} yields the following estimate on ballistic deviations of the random walker $X$:

\begin{Cor}\label{lemma: ballisticity} \label{cor: ballisticity}
    Let $\rho\in J$. 
    For any $v_\star \in (0, v^-(\rho))$, there exists $c>0$ such that for all $K \in \dN$ large enough, we have that 
    \begin{equation} \label{eq: ballisticity}
        \dP^\rho \left( \exists n \in \dN : X_n < n v_\star - K \right) \leq 2 \exp\left( -c K^{1/1000} \right).
    \end{equation}
    For any  $v_\star \in (v^+(\rho),1)$, we obtain the same bound for 
    $\dP^\rho \left( \exists n \in \dN : X_n > n v_\star + K \right)$.
\end{Cor}
\begin{proof}
We only treat the case  $v_\star \in (0, v^-(\rho))$ as the case $v_\star \in (v^+(\rho),1)$ works in the exact same fashion. 
    Let $\delta = \frac{{v}^-(\rho) - v_\star}{2}$ and $K > 4k_0^2$, where $k_0$ is the minimum value for $n$ and $L$ given in Lemma \ref{lem: generalised E_n^delta}.

    Consider the family of events $(B_n)_{n \in \dN}$ where 
    \begin{equation} \label{eq: bad n ballisticity event}
        B_n = \left\{ X^\rho_n < \left( {v}^-(\rho) - \delta \right) n \right\}.
    \end{equation}
    Note that 
    \begin{equation} \label{eq: ballisticity bad event inclusion}
        B := \left\{ \exists n \in \dN : X_n < n v_\star - K \right\} \subseteq  \bigcup_{n \geq \frac{K}{2}} B_n.
    \end{equation}
    We do not need to consider $B_n$ for $n \leq K/2$ because $v_\star \leq 1$ and the random walker only takes one step at a time, so we cannot have $X_n < n v_\star - K$ for $n \leq K/2$.
    
    Now, we will use Lemma \ref{lem: generalised E_n^delta} with $L=\lfloor \sqrt{n}\rfloor$, $f(L) = L^{1/10}$ and the above $\delta$ in order to bound the probability of $B_n$. We find that for all $n\geq K/2$ (up to increasing $K$ so that $2n^2\exp(-\lfloor \sqrt{n}\rfloor^{1/400})<\exp(-n^{1/1000})$ for all $n\geq K/2$),
    \begin{equation} \label{eq: ballisticity bad n bound}
    \begin{split}
        \dP^\rho(B_n) &\leq \dP^\rho\left( X_n < \left( {v}^-(\rho) -  \delta \right)n \right) \\
        &\leq2\left( \left\lceil \frac{n}{L} \right\rceil -1 \right)^2
            \exp \left(-f(L)^{1/40} \right)
            + 2 \exp\left(-\frac{\delta^2 n}{2^9 L}\right) \\
        &\leq \exp\left( -n^{1/1000} \right) + 2 \exp \left( -\frac{\delta^2 \sqrt{n}}{2^9} \right).
    \end{split}
    \end{equation}

    Using \eqref{eq: ballisticity bad event inclusion}, a union bound on $n\geq K/2$, \eqref{eq: ballisticity bad n bound}, and an integral bound for the sum below, for large enough $K$ (only depending on $\delta$, so that the said integral bound applies), we obtain that
    \begin{equation}
        \begin{split}
            \dP^\rho (B) &\leq \sum_{n \geq K/2} \dP^\rho (B_n)
            \leq \sum_{n \geq K/2}  \exp\left( -n^{1/1000} \right) + 2 \exp \left( -\frac{\delta^2 \sqrt{n}}{2^9} \right) \\
            &\leq \int_{K/2 \,-1}^\infty \frac{1}{2000}x^{-1+1/1000} e^{-x^{1/1000}/2}+\frac{\delta^2}{2^{11}\sqrt{x}}\exp\left(-\frac{\delta^2\sqrt{x}}{2^{10}}\right) dx
            \\
            &\leq \exp\left( -\frac{(K/2-1)^{1/1000}}{2} \right) + \exp\left( -\frac{\delta^2 (K/2-1)^{1/2}}{2^{10}} \right).
        \end{split}
    \end{equation}
    The conclusion follows. 
\end{proof}

\section{Proof of Theorem~\ref{thm:main} and Corollary~\ref{cor:APCRWsum}}\label{sec:onedensity}
\subsection{Proof of Theorem~\ref{thm:main}}\label{subsec:onedensity}
In this section, we prove Theorem~\ref{thm:main}, which specifically concerns the APCRW model defined in Section \ref{sec: model}. We first state that the APCRW satisfies all the environment properties and coupling conditions of Section~\ref{sec:defs}.

\begin{lemma}\label{lem:P1P4forAPCRW}
For all choice of $q\in [0,1]$ and $\alpha \in (0,1)$, the APCRW with drift $q$ and laziness $\lazy$ satisfies the properties~\ref{pe:markov}-\ref{pe:density} with $J=\mathbb{R}^+$, $\mu_\rho=\Poisson(\rho)^{\otimes\dZ}$ for all $\rho>0$, and for all $\eta_0\in \Sigma$, $\mathbf{P}^{\eta_0}=\mathbf{P}^{\eta_0,q}_{\text{APCRW}}$.
\end{lemma}

\begin{proof}
The proofs of~\ref{pe:markov},~\ref{pe:monotonicity} and~\ref{pe:density} are the same as in Lemma B.1\tmpdel.\tmpend in~\cite{CKKR24}. For~\ref{pe:stationary}, the only meaningful change in the proof is the way the Poisson process is split into independent copies. 
We split $\eta_0(x)$ into $\eta_0(x) = l(x) + c(x) + r(x)$, where $l(x)$ is the number of particles at $x$ at time 0 taking their first step to the left, $c(x)$ particles starting at $x$ who remain at $x$ in the first time-step, and $r(x)$ the particles starting at $x$ that take their first step to the right. We now have that $l(x) \sim \Poisson(\alpha(1-q)\rho)$, $c(x) \sim \Poisson((1-\alpha)\rho)$, and $r(x) \sim \Poisson(\alpha q \rho)$, where $(l(x), c(x), r(x))_{x \in \dZ}$ are all independent. Hence, we have that $\eta_1(x) = r(x-1) + c(x) + l(x+1) \sim \Poisson(\rho)$ for all $x \in \dZ$, with $(\eta_1(x))_{x \in \dZ}$ independent. Thus, $\eta_1 \sim \mu_\rho$. 
\end{proof}
\begin{proposition}\label{prop:conditionsforAPCRW}
   For all choice of $q\in [0,1]$ and $\alpha \in (0,1)$, the APCRW with drift $q$ and laziness $\lazy$ satisfies the conditions~\ref{pe:densitychange}-\ref{pe:nacelle}.
\end{proposition}
The proof of Proposition~\ref{prop:conditionsforAPCRW} is deferred to Appendix~\ref{sec:couplingproofs}. 

\begin{Cor}\label{cor:APCRWspeed}
For all choices of $q\in [0,1]$ and $\alpha \in (0,1)$, Theorem~\ref{thm:moregeneral} holds for the APCRW, and the map $\rho \mapsto v(\rho)$ is increasing. 
\end{Cor}

\begin{proof}
 This comes readily from Theorem~\ref{thm:moregeneral}, Lemma~\ref{lem:P1P4forAPCRW} and Proposition~\ref{prop:conditionsforAPCRW}. 
\end{proof}


In this section, we will prove \Cref{thm:main} by taking \Cref{thm:moregeneral}, and expand it to apply to all but at most one density (i.e. values of $\rho \geq 0$). 
We will do this by adapting the method found in Section 4 of \cite{HHSST15}, which was tailored for the PCRW (that is, our setting in the particular case $q=1-q=1/2$). 
The idea is to use regeneration times of the walk, using the ballistic difference between the random walker and the environment particles (which holds for all but potentially one value of $\rho$ where $v(\rho)$ could align with the speed of the environment particles). 
\\

\noindent
We define this possibly exceptional density as
\begin{equation} \label{eq: rho_0 def}
    \rho_0 = 
    \sup \left \{ \rho : {v} (\rho) < (2q-1)\nonlazy \right \}.
\end{equation}

First, we will deal with densities $\rho > \rho_0$ (which correspond to speeds faster than the environment). 
The other case, where $\rho < \rho_0$, is similar, where we change definitions in the proof below
(starting with the cones defined below in \eqref{eq: def cone up} and \eqref{eq:def cone down} below) 
to account for the gap between the movement of the random walk and the environment particles growing in the other direction. With these definitions altered, the proof remains the same.

Fix $\rho > \rho_0$. Then $v^-(\rho) > (2q-1)\nonlazy$.
Indeed, by Theorem~\ref{thm:moregeneral}, we can select $\rho', \rho'' \in (\rho_0, \rho)$ with $\rho' < \rho''$ such that $v(\rho')$ and $v(\rho'')$ exist.
By \eqref{eq: rho_0 def}, we know that $v(\rho') \geq (2q-1)\nonlazy$. 
We know that $v(\rho') < v(\rho'')$ by the monotonicity given in Corollary \ref{cor:APCRWspeed}. 
Finally, by \Cref{prop: liminf limsup}, $v(\rho'') \leq v^-(\rho)$.
Putting this together, we get that $v^-(\rho) > (2q-1)\nonlazy$.


By the same argument, we can choose some $\widebar{\rho}, \rho_\star\in (\rho_0,\rho)$ such that $\widebar{\rho} < \rho_\star$ and for $\widebar{v} = v(\widebar{\rho})$ and $v_\star = v(\rho_\star)$, 

\begin{equation}
  (2q-1)\nonlazy < \widebar{v} < v_\star < v(\rho').  
\end{equation}


Corollary~\ref{cor: ballisticity} (serving the same purpose as \cite[Eq.(1.7)]{HHSST15}) entails that our random walk will eventually be faster than $v_\star$, and hence than environment particles given our choice of $v_\star$, as desired.
However, Corollary~\ref{cor: ballisticity} only works with an offset $K$, which means we will not be able to apply it directly to a situation like the one depicted in Figure~\ref{fig: Cones}. This is the main reason why a delicate renewal structure is needed. The differences with~\cite{HHSST15} do not essentially modify the argument, but are not immediate either, which is why we make a careful review of the changes below. We first explain how this renewal structure is built. Second, we briefly sketch the proof architecture of~\cite{HHSST15}. Third, we handle the differences between their setting and ours. 
\\

\noindent
\textbf{Renewal structure.} 
In a nutshell, we first wait until the random walker takes a large number of consecutive steps to the right, in order to distance the random walker from the environment particles with which it had interacted. We will then show that the random walk does not later meet those environment particles again. Note that the time at which this happens is not a stopping time.

We will use cones to track the space-time locations of our random walk $Y_n = (X_n, n)$ and the environment particles, henceforth using $Y_n$ to refer to the space-time location of our random walk. 
The reason behind using these cones is that their angle corresponds to the speed $\widebar{v}$, which will allow them to effectively act as a buffer between the walk and the environment particles, as it is unlikely for environment particles starting at the base of the cone (or to its left) to remain in it. 
In general, we will write
\begin{equation}\label{eq: def cone up}
    \angle(x,n) = \left\{ (y,t) \in \dZ^2: t \geq n, y \geq \widebar{v}t \right\}
\end{equation} 
and
\begin{equation}\label{eq:def cone down}
    \turnangle(x,n) = \left\{ (y,t) \in \dZ^2: t \leq n, y < \widebar{v}t \right\}.
\end{equation} 
These cones are illustrated in \Cref{fig: Cones}.
\begin{figure}
    \centering
    \includegraphics[width=0.55\linewidth]{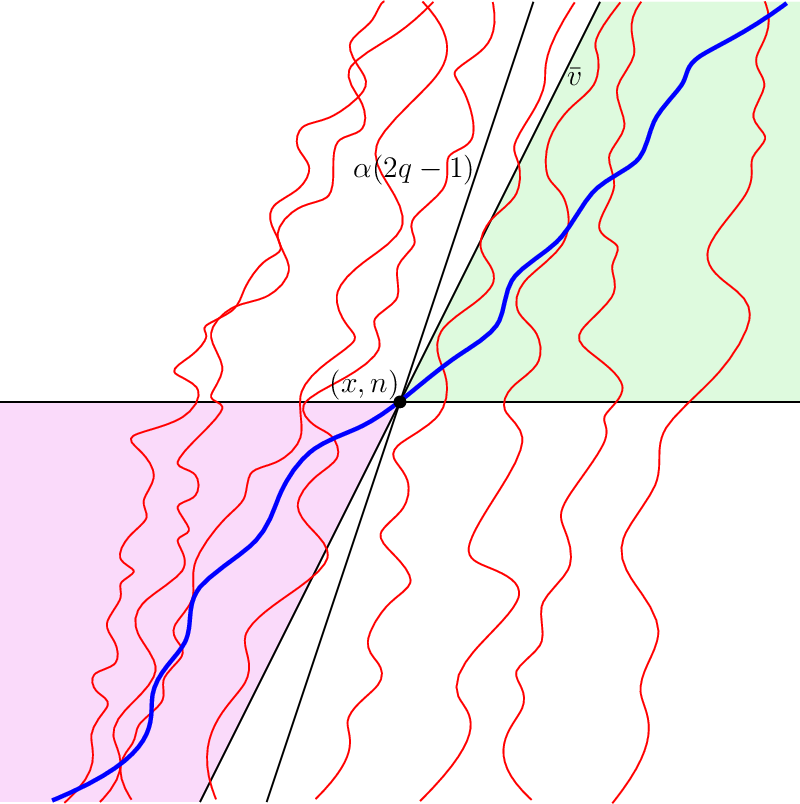}
    \caption{An illustration of trajectories and the cones $\angle(x,n)$ (top right, shaded in green) and $\turnangle(x,n)$ (bottom left, shaded in purple).
    The red trajectories are some typical examples of particles in the environment, and the blue, thicker trajectory is a possible trajectory for the walker. Note that since the angle of the cones is associated with $\widebar{v}$ and the environment particles have a speed of $\frac{2q-1}{2} < \widebar{v}$, it is unlikely for the environment particles to jump between the two cones, unlike the walker. (The dimensions are not to scale.) }
    \label{fig: Cones}
\end{figure}

We now define record times $R_k$ to be
\begin{equation} \label{eq: def record time}
    R_k = \inf \left\{ n \in \dN : (X_n, n) \in \angle((1-\widebar{v})k, 0) \right\},
\end{equation}
which represents the random walk reaching a new record distance past a line of slope $\widebar{v}$.

We will focus on those record times satisfying a set of additional conditions helping the walk leaving behind its current environment, called \textit{good} record times.  
We will then prove that good record times happen frequently, and that after each of these times, a subsequent record time shortly after has a suitably high probability of being a regeneration time. 

For a good record time to happen, we first require the random walk to leave its old downwards-facing cone $\tilde{\mathcal{C}}$ by taking a certain number $T''$ of consecutive steps to the right, upon which time it will enter a new upwards-facing cone $\mathcal{C}$, as can be seen in Figure \ref{fig: Split Cones}. 
The next condition is that particles around the random walk before it took these consecutive steps to the right do not enter $\mathcal{C}$. This will be encapsulated in the local influence field, which we require to be small enough. 
As a separate condition, we will also require particles that were between the old cone $\tilde{\mathcal{C}}$ and the new cone $\mathcal{C}$ while the random walk was making this jump between the two cones also do not enter $\mathcal{C}$. 
Finally, we will require the random walk stays in $\mathcal{C}$ for a long time.

\begin{figure}
    \centering
    \includegraphics[width = 0.69\linewidth]{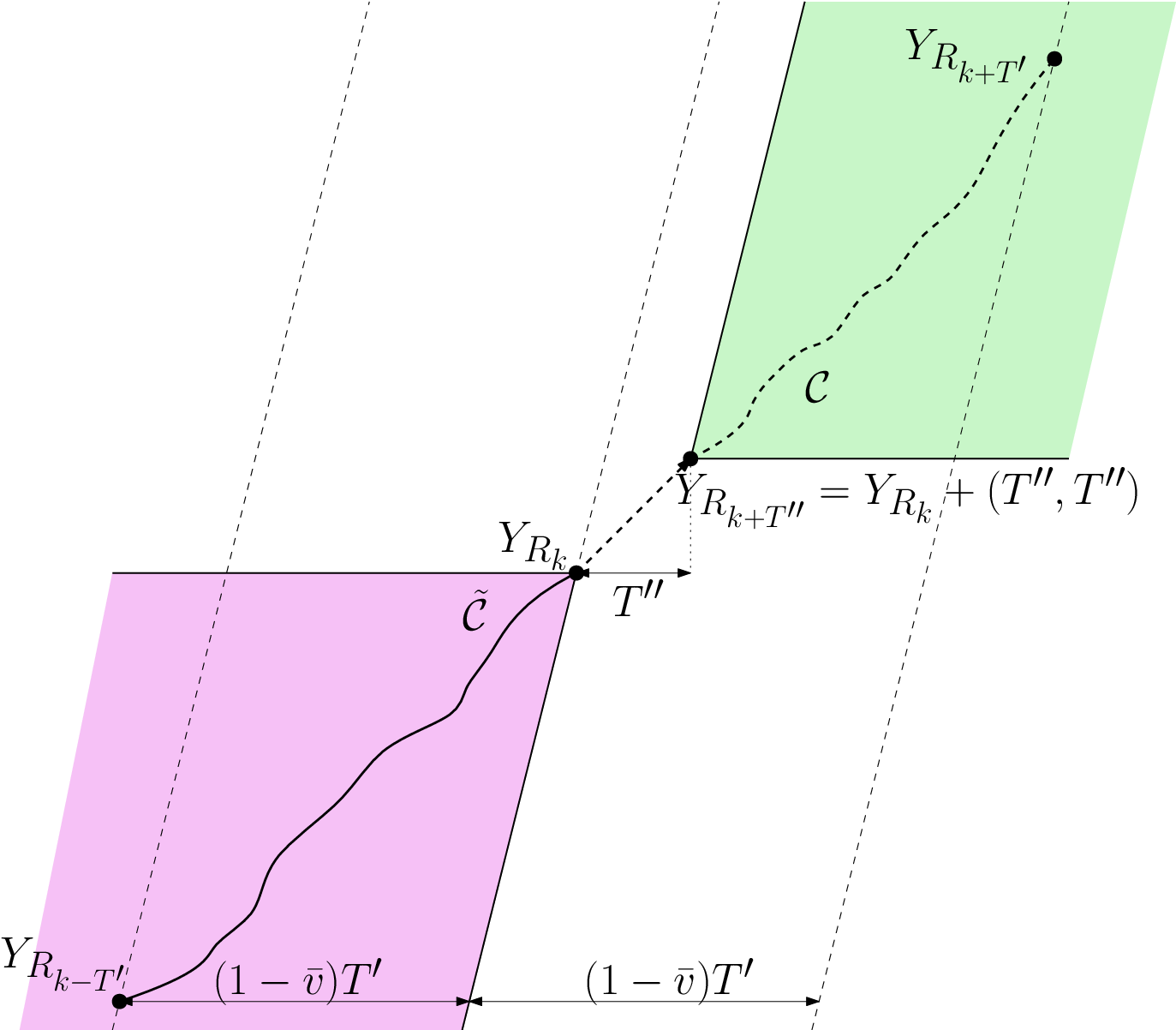}
    \caption{A depiction of a random walk such that $R_k$ is a good record time.}
    \label{fig: Split Cones}
\end{figure}



A word about the influence field $h(x,n)$: it is defined as the minimum distance apart two cones (one downwards-facing, one upwards-facing) must be for no particles in the environment to make a jump between them, and the local influence field $h^T(x,n)$, which adds in the extra condition that the trajectories of these particles does not intersect the cone $\turnangle(x-\lfloor (1-\widebar{v})T'\rfloor, n)$ (see \cite[(4.21),(4.30)]{HHSST15} for details). 
As said previously, at a good record time, we are looking for this local influence field to be sufficiently small. 
\\


\textbf{Proof structure from \cite{HHSST15}.}
In \cite{HHSST15}, Theorem 4.1 ensures that regeneration times constructed above are indeed renewal times, cutting the trajectory of the walk into i.i.d. bits (and its standard proof will work in the same way in our case). 
\\
Then, Theorem 4.2 shows that the first regeneration time has all its moments finite, which is enough for our purpose. To this end,  
Lemmas 4.3 and 4.4 give exponential tails to the influence field and local influence field, respectively. 
Lemma 4.5 shows that 
once the random walk enters a cone, it has a probability bounded away from 0 to enter a certain subsequent cone in a relatively short amount of time. Geometrically, this is interpreted as entering a parallelogram from the bottom left corner and  exiting it from the right side. 
Proposition 4.6 combines these and bounds the probability that none of the $R_k$'s, $k \in [1, T]$, is a good record time. 
Finally, in the proof of Theorem 4.2, from a good record time $R_k$, we obtain a regeneration time at time $R_k+T''$ (where $T'' = \lfloor \delta \log T \rfloor$, with $\delta = \frac{1}{4\log(\frac{1}{\pempty \wedge \pfull})}$, is actually the distance between the two cones in the definition of a good record time), by adding in the high probability conditions that everywhere in some large surrounding box, the influence field is not too large and the deviations of a random walk (starting from any given point in the aforementioned box) are not too large in the wrong direction. We can bound the probabilities of these two events by the exponential tails on the size of the influence field and Corollary \ref{cor: ballisticity} for the walk's deviations, then taking crude union bounds.
\\


\textbf{Differences with \cite{HHSST15}.}
Lemma 4.3 from \cite{HHSST15} -- which gives an upper bound for the probability of the influence field $h(x)$ being larger than a certain size $l$ -- works the same way, slightly changing equation (4.25) (and therefore also (4.26) and (4.27)) to account for the drift of the environment particles, and having $Z_n(\omega)$ be distributed like a two-sided asymmetric random walk starting at $x$. 
In particular, the bound found via Azuma's inequality in (4.25) becomes
\begin{equation}
    \mu (W_{y \leftrightarrow y'} ) = P_0 (Z_{n' - n} = x' - x) 
    \leq \exp \left( - \frac{(x' - x - \nonlazy(2q-1)(n'-n)}{2 (n' - n)} \right). 
\end{equation}
Thus (4.26) becomes 
\begin{equation}
\begin{split}
        &\dP (h(0) > l) \leq 
        \\
        &\rho \sum_{(x, n) \in \turnangle(0,0)} \sum_{(x', n') \in \angle(l,l)} 
    \exp \left( - \left(\frac{\widebar{v} - \nonlazy(2q-1)}{2} \right) \left( x' - x - \nonlazy(2q-1)(n'-n) \right) \right).
\end{split}
\end{equation}
Then (4.27) is changed accordingly. Since $\widebar{v} > \nonlazy(2q-1)$, this is not an issue, and the lemma itself is unchanged. 

Lemma 4.4 from \cite{HHSST15} -- which gives a similar bound to the one found in Lemma 4.3 for the local influence field $h^T(x)$ -- follows immediately from Lemma 4.3. This remains unchanged. 

Lemma 4.5 from \cite{HHSST15} itself remains similar, but the definition of the parallelogram we use changes as follows. 

\begin{figure}
    \centering
    \includegraphics[width=0.5\linewidth]{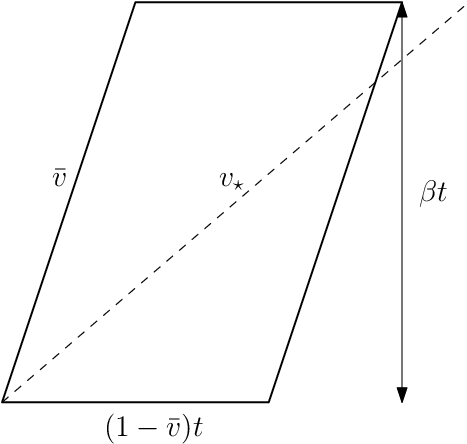}
    \caption{The parallelogram $\mathcal{P}_t(y)$, which has slope $\widebar{v}$, length $(1-\widebar{v})t$ and height $\beta t$. Note that $\beta$ is such that a line with slope $v_\star$ will exit from the right side of the parallelogram.}
    \label{fig: parallelogram}
\end{figure}

Let $\mathcal{P}_t(y)$ denote the space-time parallelogram with slope $\widebar{v}$, (space) length $(1-\widebar{v})t$, and (time) height $\beta t$, where we take $\beta=\beta(\widebar{v},v_\star,\rho)$ large enough so that a line with slope (speed) $v_\star$ starting from the bottom left corner will exit from the right side of the parallelogram. 
$\mathcal{P}_t(y)$ is depicted in Figure~\ref{fig: parallelogram}.
These selections are made so that the random walk starting in the bottom left has a high probability to exit from the right, and when it does, it is separated from the environment particles that were initially on its left by the full width of the parallelogram (since the environment particles have a speed slower than the slope $\widebar{v}$ of the parallelogram).
In \cite{HHSST15}, $\beta$ is easier to define explicitly, with $\widebar{v}$ being a know function of $v_\star$, whereas in our case, 
leaving $\beta$ implicit is simpler as we also require $\widebar{v} > (2q-1)\alpha$. The height of the parallelogram is such that we have sufficient separation between lines of slope $\widebar{v}$ and $v_\star$, and $\beta$ is changed accordingly.
The proof of Lemma 4.5 works the same way, using Corollary \ref{cor: ballisticity} for the ballisticity condition given by (1.7) in \cite{HHSST15}. 

Lemma 4.6 in \cite{HHSST15} states that the probability that for $k \leq T$, the probability that no record time $R_k$ is a good record time is exponentially small in $T$, i.e. that
\begin{equation}
    \dP \left( \forall k \in [1, T], R_k \text{ is not a good record time} \right)\leq e^{-c T^{1/2}}.
\end{equation}
This is proved by examining the probability of a single record time being a good record time. 

This change impacts the proofs of equations (4.47) and (4.50) in the proof of Lemma 4.6 from \cite{HHSST15}.
For the proof of (4.47) at (4.51), the first equality is an unchanged union bound. Then, for the next inequality, the purpose of the parallelogram -- to separate environment particles from the random walker -- is unchanged, and one can still perform a union bound (whence the sum) over the possible exit points on the right boundary of the parallelogram. Then, in the final line of (4.51), $c_5$ in~\cite{HHSST15} (which should actually read $c_5/v'$ there)  becomes $\beta c_5$, as we are summing over the right boundary of the parallelogram. 

The changes to the parallelogram do not change how it is used in the proof of (4.50), at (4.53). Indeed, if the random walk $Y^{Y_{R_{k+T''}}}$ exits the parallelogram $\mathcal{P}_{T'}(Y_{R_{k + T''}})$ through the right side, then it means that the random walk remains in the cone $\angle(Y_{R_{k+T''}})$ until time $R_{k + T'}$ (since that time is when it exits the parallelogram - with $T' = \left\lfloor T^\epsilon \right\rfloor$ for $\epsilon = \frac{1}{4} \left(\Cl[c]{epsilon scale for local influence field} \delta \wedge 1 \right)$).  

The other conditions of being a good record time that are checked in Lemma 4.6 in \cite{HHSST15}, and the completion of the proof of Theorem 4.2, are unaffected. 

Then, completing the proof of Theorem 4.2 from \cite{HHSST15} is unchanged, giving that  $\dE \left[ e^{c \log^\gamma \tau} \right] < \infty$, where $\tau$ is the first regeneration time. 

From there, we can use Theorem 4.1 from \cite{HHSST15} to conclude that Theorem 4.2 is true for other (non-first) regeneration times.
\\
\\
To conclude, we have just established the existence of integrable renewal times for $X$ under $\mathbb{P}^\rho$, which implies the existence of $v(\rho)>(2q-1)\alpha$ (if $\rho >\rho_0$, and  $v(\rho)<(2q-1)\alpha$ if $\rho<\rho_0$) such that
\begin{equation}
    \frac{X_n}{n} \underset{n\rightarrow\infty}{\longrightarrow} v(\rho)
\end{equation}
$\mathbb{P}^\rho$-a.s. Note that by Corollary \ref{cor:APCRWspeed}, we retain the monotonicity of $v(\cdot)$ in $\rho$, thus completing the proof of Theorem \ref{thm:main}.

\subsection{Proof of Corollary~\ref{cor:APCRWsum}}\label{subsec:apcrwsum}
Fix $N\in \mathbb{N}$ and consider a mixture of $N$ APCRWs as in the statement of the corollary. From Propositions~\ref{prop: apcrw superposition} and~\ref{prop:conditionsforAPCRW}, this new environment satisfies our conditions with the slight modifications mentioned in Section~\ref{subsec:superpos}. This triggers two kinds of changes that we need to track down in our arguments: some constants change in the probability bounds of the non-desirable events in these conditions (one has to take the worst value of $C_1, C_2, \ldots$ over the $N$ different environments, and an extra $N$ factor arises from a union bound over all environments). More subtly, as we now require particles of \textit{each} type to have the right empirical density in~\ref{pe:densitychange}-\ref{pe:nacelle}, the events that the bespoke condition cannot be applied in a given space-time frame could a priori have a larger probability (with the same kind of changes in the bounds). All these updates turn out to be harmless. Let us illustrate this here for Section~\ref{sec:couplings} - the same applies for Section~\ref{sec:speed} (and~\cite{CKKR24} for the monotonicity of $v$), we leave out the details. 
\\
In Section~\ref{sec:couplings}, $G_1$ at~\eqref{eq: many to one G1 def} must now be the intersection of the event described applied to each of the $N$ types of particles. Hence at~\eqref{eq: many to one G1 prob}, the RHS gets a prefactor $N$, and $\Cr{densitydev}$ has to be small enough so that~\ref{pe:density} holds for each of the $N$ types of particles (which leaves $\Cr{densitydev}>0$ as we take the minimum of finitely many positive quantities). Similarly for the recoupling event $G_2$ at~\eqref{eq: many to one G2 def}, and~\eqref{eq: many to one G2 prob} sees its RHS gaining a factor $N$ and $\Cr{SEPcoupling2}$ being potentially altered. Matters are even simpler for $G_3$ at~\eqref{eq: many to one G3 def} and~\eqref{eq: many to one G3 prob}, as the probability of~\ref{pe:drift} failing is precisely zero for the APCRW. All in all, this does not change the final bound at~\eqref{eq: many to one union bound proba}. 
\\
Lastly, when all the drifts $d_i$ are equal, one can readily check that the renewal structure of the previous section can be implemented in the same way, with $(2q-1)\alpha$ replaced by $d_1$.

\appendix
\section{Proof of Proposition~\ref{prop:conditionsforAPCRW}: coupling conditions for the APCRW}\label{sec:couplingproofs}

In this appendix, we first present some technical estimates related to the APCRW environment, and then use those estimates to help prove \ref{pe:densitychange}~-~\ref{pe:nacelle}.

\subsection{Technical estimates: heat kernel and soft local times}
In this section, we write $\widebar{q_{n}}$ for the discrete time heat kernel of a non-lazy asymmetric random walk taking a step up with probability $q$, and a step down with probability $1-q$. That is, $\widebar{q_{n}}(z)$ is the probability that this walk, starting at $0$, reaches $z \in \integers$ after $n\geq 0$ steps. 
\\
We first give an estimate for the heat kernel itself, in the spirit of~\cite[Proposition 2.5.3]{LL10} (we could not locate this adaptation, which does not bear any notable difficulty, in the existing literature). This would suffice to control the single site marginals $\eta_t(z)$ of an environment $\eta$ after some time $t$, given a possibly deterministic but regular enough initial condition $\eta_0$. However, the crux of most of our conditions relies on controlling the \textit{joint} distribution of $\eta_t(z)$ for $z$ ranging over an appropriate interval. This is much more delicate, and we use the framework of the so called \textit{soft local times} from~\cite{HHSST15,PopovTeixeira}, which relies on a clever Poissonization of our system.
\begin{lemma} \label{Lemma: Heat Kernel}
	Let $(a_n)_{n\geq 0}$ be an arbitrary positive sequence converging to 0. For $n\geq 0$ and $z\in \integers$, let $w = z - (2q-1)n$. As $n\rightarrow \infty$, we have that uniformly in $z\in [(2q-1)n-na_n,(2q-1)n+na_n]$:
	\begin{equation}
		\widebar{q_{2n}}(2z) = 
		\sqrt{\frac{n}{\pi \left(n^2 - z^2 \right)}} 
			\exp \left( - \frac{w^2}{n} \left( \frac{1}{4q(1-q)} \right) + O\left( \frac{w^3}{n^2} + \frac{w}{n} \right) \right)
	\end{equation}
    and the same holds for $\widebar{q_{2n+1}}(2z+1)$.
\end{lemma}

\begin{proof}
We only give the proof for $\widebar{q_{2n}}(2z)$, as that for $\widebar{q_{2n+1}}(2z+1)$ can be deduced mutatis mutandis. 
By the usual binomial and Stirling's formulae, we have
\begin{equation} 
\begin{split}
	\widebar{q_{2n}}(2z) &= {2n \choose n+z} q^{n+z} \left(1-q \right)^{n-z} \\
	&= \frac{\left(2n\right)!}{(n+z)! (n-z)!} q^{n+z} \left(1-q \right)^{n-z} \\
	&= q^{n+z} (1-q)^{n-z} \frac{\sqrt{2\pi} (2n)^{2n + \frac{1}{2}} e^{-2n} [1 + O(\frac{1}{n}) ]}{2\pi (n+z)^{n+z+\frac{1}{2}} (n-z)^{n-z+ \frac{1}{2}} e^{-2n} [1 + O(\frac{1}{n})]} \\
	&= \sqrt{ \frac{n}{\pi \left(n^2 - z^2 \right)}} \left(\frac{4n^2 q (1-q)}{n^2 - z^2} \right)^n \left( \frac{q(n-z)}{(1-q)(n+z)} \right)^z   \left(1 + O(n^{-1})\right) \\
	&= \sqrt{ \frac{n}{\pi \left(n^2 - z^2 \right)}} \left(1 + \frac{w(z + (2q-1)n)}{n^2 - z^2} \right)^n \left( 1 - \frac{w}{(1-q)(n+z)} \right)^z   \left(1 + O(n^{-1})\right).
\end{split}
\end{equation}
Note that the $O(n^{-1})$ from Stirling's formula is uniform in $z$ since by our constraints on $z$, $n-z,n+z\geq (1-q)n$ for $n$ large enough. Then, we can perform the following development, using that $w/n=o(1)$ uniformly in $z$:

\begin{equation} 
\begin{split}
	\widebar{q_{2n}}(2z)
	&= \sqrt{ \frac{n}{\pi \left(n^2 - z^2 \right)}} 
		\exp \left(\frac{nw(z + (2q-1)n)}{n^2 - z^2} \right)
		\exp \left( -\frac{wz}{(1-q)(n+z)} \right) \\
		&\qquad \qquad
		\exp \left(-\frac{nw^2(z + (2q-1)n)^2}{2\left(n^2 - z^2\right)^2} + O\left( \frac{n w^3 \left(z + (2q-1)n \right)^3}{\left( n^2 - z^2 \right)^3 } \right) \right) \\
		& \qquad \qquad
		\exp \left(-\frac{w^2z}{2(1-q)^2(n+z)^2} + O\left( \frac{w^3  z}{\left( 1-q \right)^3 \left( n+z \right)^3} \right) \right) \left(1 + O\left(n^{-1} \right)\right) .
\end{split}
\end{equation}

We will now write this in $n$ and $w$, not $z$, in particular looking for terms of order at least $w^2/n$
. The first exponential term is

\begin{equation} 
\begin{split}
	\frac{nw \left( z + (2q-1)n \right)}{n^2 - z^2} 
	&= \frac{nw \left( w + 2(2q-1)n \right)}{n^2 - \left(w + (2q-1)n \right)^2} \\
	&= \frac{nw^2 + 2(2q-1)n^2w}{n^2\left(1-(2q-1)^2\right) -w^2 - 2(2q-1)nw} \\
	&= \frac{\frac{w^2}{n} + 2(2q-1)w}{\left( 4q -4q^2 \right) \left(1 - \frac{4q-2}{4q - 4q^2} \frac{w}{n} - \frac{1}{4q-4q^2}\frac{w^2}{n^2} \right)} \\
	&= \frac{\frac{w^2}{n} + 2(2q-1)w}{4q \left( 1 -q \right)} \left(1 + \frac{4q-2}{4q - 4q^2} \frac{w}{n} + o\left( \frac{w}{n} \right) \right) \\
	&= w \left( \frac{2q-1}{2q(1-q)} \right)   
		+   \frac{w^2}{n} \left( \frac{1}{4q(1-q)}  +  \frac{(2q-1)^2}{4q^2(1-q)^2}  \right) + O \left(\frac{w^3}{n^2} \right).
\end{split}
\end{equation}

Similarly, for the other exponential terms, we find that
\begin{equation}
	\frac{-wz}{(1-q)(n+z)} = w \left(\frac{-(2q-1)}{2q(1-q)} \right) + \frac{w^2}{n} \left( \frac{-1}{2q(1-q)} + \frac{2q-1}{4q^2(1-q)} \right) +
		O\left(\frac{w^3}{n^2}\right),
\end{equation}
\begin{equation}
	-\frac{nw^2(z + (2q-1)n)^2}{2\left(n^2 - z^2\right)^2} + O\left( \frac{n w^3 \left(z + (2q-1)n \right)^3}{\left( n^2 - z^2 \right)^3 } \right)
	=- \frac{w^2}{n} \left( \frac{4 (2q-1)^2}{2\left(4q(1-q)\right)^2} \right) + O \left(\frac{w^3}{n^2} + \frac{w}{n} \right) ,
\end{equation}
and 
\begin{equation}
	-\frac{w^2z}{2(1-q)^2(n+z)^2} + O\left( \frac{w^3  z}{\left( 1-q \right)^3 \left( n+z \right)^3} \right)
	= -\frac{w^2}{n} \left( \frac{2q-1}{2(2q)^2(1-q)^2} \right) + O \left(\frac{w^3}{n^2} + \frac{w}{n} \right).
\end{equation}

So
\begin{equation} 
\begin{split}
	\widebar{q_{2n}}(2z) 
	&= \sqrt{ \frac{n}{\pi \left(n^2- z^2 \right)}} 
		\exp \biggl( w \left(\frac{2q-1}{2q(1-q)} - \frac{2q-1}{2q(1-q)} \right)   \\
		& \qquad \qquad 
		+ \frac{w^2}{n}  \biggl( \frac{1}{4q(1-q)}  +  \frac{(2q-1)^2}{4q^2(1-q)^2}  -  \frac{1}{2q(1-q)}  + \frac{2q-1}{4q^2(1-q)}    \\
		& \qquad \qquad \qquad
		- \frac{4 (2q-1)^2}{2\left(4q(1-q)\right)^2}  - \frac{2q-1}{2(2q)^2(1-q)^2}  \biggr) 
		+ O \left(\frac{w^3}{n^2} + \frac{w}{n} \right)+O\left(\frac{1}{n}\right) \biggl)   \\
	&= \sqrt{\frac{n}{\pi \left(n^2 - z^2 \right)}} 
		\exp \left( - \frac{w^2}{n} \left( \frac{1}{4q(1-q)} \right) + O\left( \frac{w^3}{n^2} + \frac{w}{n} \right) \right) .
\end{split}
\end{equation}

\end{proof}

We need a few definitions for the next two lemmas, following definitions in \cite{CKKR24}.
First, we define $\eta^{\rho \pm \varepsilon}$. For $z \in I_t := [t, H-t]$, let $\eta^{\rho \pm \varepsilon} = \Lambda \left( \{z\} \times (0, \rho \pm \varepsilon] \right)$, where $\Lambda$ is a Poisson point process on $\dZ \otimes \posreals$ with intensity $1 \otimes \lambda$, where $1$ is the counting measure, and $\lambda$ the Lebesgue measure. 
For $z \notin I_t$, let $\eta^{\rho \pm \varepsilon} (z) = \Poisson(\rho \pm \varepsilon)$ independently. 
For a given initial configuration $\eta_0 \in \Sigma$, let $(x_i)_{i \geq 1}$ be an ordering of the locations of particles in $\eta_0$ (NB: this will likely include repeated locations). 
For $i \geq 1$ and $z \in [-t, H+t]$, let $g_i(z) = {q_t}(x_i, z)$, where $q$ is the heat kernel of an environment particle, that is an asymmetric random walk with laziness parameter $\lazy$ and probability to take a step to the right $\nonlazy q$. 
Then recursively define $(\xi_i)_{i \in \dN}$ by setting $\xi_1 := \sup \{ t \geq 0 : \bigcup_{z \in \dZ} \Lambda ( \{z\} \times (0, t g_1(z) ] ) = 0 \} $, and letting 
\begin{equation}
    \xi_i := \sup \{ t \geq 0 : \bigcup_{z \in \dZ} \Lambda ( \{z\} \times (0, \sum_{j = 1}^{i-1} \xi_j g_j(z) + t g_{i}(z) ] ) = i-1 \}
\end{equation}
for $i \geq 2$. 

We will then define the soft local time $G_{\eta_0}(z)$ for $z \in [-t, H+t]$ as
\begin{equation} \label{eq: soft local time def}
    G_{\eta_0}(z) = \sum_{i \geq 1} \xi_i {q_t}(x_i, z).
\end{equation}

\begin{lemma} \label{Lemma: Soft Local Times}
	There exist constants $0 < \Cl{lower bound for t in sltc}, \Cl[c]{upper bound on t with epsilon in sltc}, \Cl[c]{prob bound for G in sltc}, \Cl{slt exp tightness} < \infty$ such that the following holds. Let $\rho \in (0,K)$ and $\varepsilon \in (0, (K-\rho \wedge \rho \wedge 1)$, and let $H, \ell, t \in \dN$ be such that	$\Cr{lower bound for t in sltc}\ell ^2 < t < \frac{H}{2}$ and $(\rho + \varepsilon) \ell \sqrt{\frac{\log t}{t}} < \Cr{upper bound on t with epsilon in sltc} \varepsilon$. Then for all $z\in I_t$, 
		\begin{equation} \label{eq: soft local time bounds}
		\begin{split}
			\left( \rho - \frac{\varepsilon}{2} \right) \left(1 + \Cr{slt exp tightness}\ell \sqrt{\frac{\log t}{t}} \right) \leq \Ecoupling \left[ G_{\eta_0} (z) \right]; \\
			\Ecoupling \left[ G_{\eta_0} (z) \right] \leq \left( \rho + \frac{\varepsilon}{2} \right) \left(1 + \Cr{slt exp tightness}\ell \sqrt{\frac{\log t}{t}} \right).
		\end{split}
		\end{equation}
\end{lemma}
The proof for this will roughly follow the proof of \cite[Lemma B.4]{CKKR24}, with a few technical differences due to the use of the heat kernel of an asymmetric random walk instead of a symmetric one.
\begin{proof}
Fix parameters satisfying the above and let $z \in [ - H + t, H - t ]$. Write $c_t = \sqrt{t \log t}$, and let $d_t = \alpha t (2q-1)$ be the expected position at time $t$ of an environment particle starting from $z$ (or in other words, the drift). 
\\	
We now cover $[\lfloor z + d_t - c_t\rfloor, \lfloor z + d_t + c_t\rfloor]$ by a family $\left( I_i \right)_{1 \leq i \leq m}$ of intervals containing $\ell$ integers each, all being pairwise disjoint with the potential exception of $I_1$ and $I_2$ which may overlap ($m$ being the bespoke cardinality). 
	
	We will only consider the upper bound, as the lower bound is shown in a similar but slightly simpler way (only the bulk term of the sum below being considered). We write
	\begin{equation}\label{eq: EQG decompo}
	\begin{split}
		\Ecoupling \left[ G_\eta (z) \right] &= \Ecoupling \left[ \sum_{i \geq 1}  \xi_i q_t(x_i, z)  \right]
			= \sum_{i \geq 1} \Ecoupling \left[ \xi_i q_t(x_i, z) \right]  
			= \sum_{i \geq 1} q_t (x_i, z) \\
			&\leq \sum_{i = 1}^{m} 
				\left( \rho + \frac{\varepsilon}{2} \right) |I_i| \max_{x_j \in I_i} \left( q_t (x_j, z) \right) 
				\\
                &\qquad+\left(\rho+\frac{\varepsilon}{2}\right)\ell \left(\sum_{x > z + \lfloor c_t+d_t \rfloor} q_t (x, z)
				+ \sum_{x < z - \lfloor c_t-d_t \rfloor} q_t (x, z)\right)
	\end{split}
	\end{equation}
	
	We first look at the last two terms. Applying Azuma's inequality to the trajectory of an environment particle and controlling its probability to end outside of $(d_t-c_t+1,d_t+c_t-1)$ in time $t$, we obtain as in~\cite[(B.11)]{CKKR24}:
    
	\begin{equation}\label{eq:heat kernel dev2}
	\begin{split}
		\sum_{\substack{x \geq z + \lfloor c_t \rfloor \\ \text{or }\\ x \leq z - \lfloor c_t \rfloor}} q_t (x, z) 
			&\leq 2 \exp \left( - \frac{(c_t - 1)^2}{2t} \right) 
			\leq 2 \exp \left( - \frac{\log t}{2} + \sqrt{ \frac{\log t}{t}} \right),
	\end{split}
	\end{equation}
	and this entails the desired control. We can then handle the first term
	\begin{equation} \label{eq:expectation estimate lemma first term split}
	\begin{split}
		\sum_{i = 1}^{m} 
					\left( \rho + \frac{\varepsilon}{2} \right) &|I_i| \max_{x_j \in I_i} \left( q_t (x_j, z) \right) \\
				& \leq \left( \rho + \frac{\varepsilon}{2} \right) \sum_{i = 1}^{m} 
					 \sum_{y \in I_i} q_t (y, z)
					 + \left( \rho + \frac{\varepsilon}{2} \right) \sum_{i = 1}^{m} 
					 \sum_{y \in I_i} \max_{x \in I_i} \left( q_t (x,z) - q_t (y, z) \right)
	\end{split}
	\end{equation}
    in the same was as in~\cite{CKKR24}, once we replace (B.13) and (B.15) respectively by
	\begin{equation} \label{eq: max of heat kernel}
		\max_{y,z \in \integers} q_t (y, z) \leq C t^{- \frac{1}{2}}
	\end{equation}
	for some constant $C$ and all $t\geq 1$, and
    \begin{equation}\label{eq:qtxzqtxy}
        q_t (x, z) 
			\leq q_t(y,z) \left(1 + C\frac{\ell c_t}{t} \right)
				+ 2 \exp \left( -8 \log t \right)
    \end{equation}
    for all $x,y$ in a common interval $I_i$. 
    \\
    
    \textbf{Proof of~\eqref{eq: max of heat kernel}}
    Such an estimate is standard for symmetric random walks (see e.g. \cite[Proposition 2.4.4]{LL10}). In our setting, we leverage Lemma~\ref{Lemma: Heat Kernel} and discriminate on the number $N_t$ of steps taken by the lazy biased random walk by time $t$. By translation invariance and~\eqref{eq:heat kernel dev2}, we only have to show
    \begin{equation}
       \max_{y \in [d_t-c_t, d_t+c_t]\cap \integers} q_t (0,y) \leq C t^{- \frac{1}{2}}. 
    \end{equation} 
    For any integer $y \in [d_t-c_t, d_t+c_t]$, we have 
    \begin{equation}
        q_t(0,y)\leq \sum_{k\geq 0} \widebar{q_{k}}(0,y)\dP(N_t=k) \leq \exp(-t^{1/6})+\sum_{k\in [\nonlazy t- t^{3/5}, \nonlazy t+t^{3/5}]} \widebar{q_{k}}(0,y)\dP(N_t=k)
    \end{equation}
   for $\Cr{lower bound for t in sltc}$ and thus $t$ large enough (only depending on $\alpha$) by another standard application of Azuma's inequality, using that $N_t\sim \text{Bin}(t, \nonlazy)$. Letting $w=y-(2q-1)k$, we have 
   $\vert w \vert \leq t^{3/5}+c_t \leq k^{2/3}$ for all $k$ considered, provided again that $\Cr{lower bound for t in sltc}$ is large enough, and thus Lemma~\ref{Lemma: Heat Kernel} yields that
\begin{equation}
\begin{split}
    q_t(0,y)&\leq  \exp(-t^{1/6})
    \\
    +&\sum_{k\in [\nonlazy t- t^{3/5}, \nonlazy t+t^{3/5}]} \sqrt{\frac{k}{\pi \left(k^2 - y^2 \right)}} 
			\exp \left( - \frac{w^2}{k} \left( \frac{1}{4q(1-q)} \right) + O\left(1\right) \right)  \dP(N_t=k)    
\end{split}
    \end{equation}
   where the $O(1)$ is uniform in the choice of $y$ and $k$. Since $k/(k^2-y^2)\geq c/k\geq c/(2\nonlazy t)$ uniformly in $k$ and $y$ (for $\Cr{lower bound for t in sltc}$ and $t$ large enough) for some constant $c$ uniquely depending on $q$ and $\alpha$,~\eqref{eq: max of heat kernel} follows. 
   \\

	
\textbf{Proof of~\eqref{eq:qtxzqtxy}.} The working is similar to the proof of~\cite[(B.15)]{CKKR24}, translating this in estimates with the discrete heat kernel, which brings the necessity to distinguish parity cases. 
Suppose first that $|x- z|$ and $|y - z|$ are both even. We have
	\begin{equation}\label{eq:decompo qtxz}
	\begin{split}
		q_t (x, z) \leq \sum_{n = \left\lfloor \frac{\nonlazy t}{2} - c_t \right\rfloor}^{\left\lceil \frac{\nonlazy t}{2} + c_t \right\rceil}\dP(N_t = 2n) \widebar{q_{2n}} (y, z) \frac{\widebar{q_{2n}} (x, z)}{\widebar{q_{2n}}(y,z)}  + \dP \left( \left| N_t - \nonlazy t \right| > 2 c_t \right)
	\end{split}
	\end{equation}
	
	Letting $w = x - z + 2n(2q-1)$ and $w' = y - z + 2n(2q-1)$, for $n \in \left[ \left\lfloor \nonlazy\frac{t}{2} - c_t \right\rfloor, \left\lceil \nonlazy\frac{t}{2} + c_t \right\rceil \right]$, we have $\vert w\vert, \vert w'\vert\leq 1 + 3c_t$. We arrive at this result because with these bounds on $n$, $|2n(2q-1) - d_t| \leq 2c_t$. Thus, we can again apply Lemma~\ref{Lemma: Heat Kernel} and obtain (uniformly in the choice of $x,y$ and $I_i$): 
	\begin{equation} \label{eq: slt disc heat kernel ratio}
		\frac{\widebar{q_{2n}} (x, z)}{\widebar{q_{2n}}(y,z)} 
			= \sqrt{ \frac{\frac{n}{\pi \left( n^2 - (x-z)^2 \right)}}{\frac{n}{\pi \left( n^2 - (y-z)^2 \right)}}}
				\exp \left( - \left(\frac{w^2}{n} - \frac{w'^2}{n} \right) \left( \frac{1}{4q(1-q)} \right) 
				+ O\left( \frac{w^3 - w'^3}{n^2} + \frac{w - w'}{n} \right) \right).  
	\end{equation}
Note that the terms in the exponential are $O\left( \ell \sqrt{(\log t)/t}\right)$ since (for $\Cr{lower bound for t in sltc}$ large enough)
\begin{center}
    $\vert (w^2-w'^2)/n\vert =\vert x-y\vert \times \vert w+w'\vert/n\leq  \frac{4 \ell (1 + c_t)}{\nonlazy t - 2 c_t}$ with $c_t=\sqrt{t \log t}$, 
\end{center}
$\vert w- w'\vert /n\leq \frac{2 \ell}{\nonlazy t - 2 c_t}$ and  
    \begin{center}
    $\vert (w^3-w'^3)/n^2\vert =\vert x-y\vert \times \vert w^2+ww'+w'^2\vert/n^2\leq l\times 3(1+ 3c_t)^2/n^2\leq \frac{6 \ell (1 + 3 c_t)^2}{\nonlazy^2 t^2 - 4 c_t^2}$.  
\end{center}
	
	Now, let us look more closely at the square root from \eqref{eq: slt disc heat kernel ratio}. Writing $\bar{x} = x - z\in [d_t-c_t-1, d_t+c_t+1]$, we note that $n,\bar{x},n-\bar{x}=\Theta(n)$ with constants that only depend on $\nonlazy$ and $q$. We then deduce
	\begin{equation}
	\begin{split}
		\sqrt{ \frac{ n^2 - (y-z)^2 }{ n^2 - (x-z)^2 } } &= \sqrt{ \frac{ n^2 - \left( \bar{x} - (x - y) \right)^2}{n^2 - \bar{x}^2}} \\
		&= \sqrt{ \frac{n^2 - \bar{x}^2 + 2 \bar{x} (x-y) - (x-y)^2}{n^2 - \bar{x}^2}} \\
		&= \sqrt{1 + \frac{2\bar{x} (x-y) - (x-y)^2}{n^2 - \bar{x}^2}} \\
		& = \sqrt{ 1 + O\left(\frac{\ell}{\sqrt{n}}\right) },
	\end{split}
	\end{equation}
	where we used again that $\vert x-y\vert\leq \ell$. Together with the control of the exponential terms, this yields a constant $C=C(\nonlazy,q)$ such that for $\Cr{lower bound for t in sltc}$ large enough, 
	\begin{equation}
		\frac{\widebar{q_{2n}} (x, z)}{\widebar{q_{2n}}(y,z)} 
			\leq  1  + C\frac{\ell c_t}{t}. 
	\end{equation}
	Combining this with~\eqref{eq:decompo qtxz} and using Azuma's inequality to handle the last term, we have that
	\begin{equation}
	\begin{split}
		q_t (x, z) 
			&\leq \sum_{n = \left\lfloor \frac{\nonlazy t}{2} - c_t \right\rfloor}^{\left\lceil \frac{\nonlazy t}{2} + c_t \right\rceil} \dP(N_t = 2n) 
				\widebar{q_{2n}} (y, z) \exp \left( C \frac{\ell c_t}{n} \right)
				+ \dP \left( \left| N_t - \nonlazy t \right| > 2 c_t \right) \\
			&\leq q_t(y,z) \left(1 + C\frac{\ell c_t}{t} \right)
				+ 2 \exp \left( -8 \log t \right).
	\end{split}
	\end{equation}
	If $|x - z|$ and $|y-z|$ are both odd, then a very similar computation works, swapping the number of steps to be $2n+1$. Finally, if one of $|x-z|$ and $|y - z|$ is even and the other is odd -- say $|x-z|$ is even, without loss of generality -- then for $n$ between $\left\lceil \frac{\nonlazy t}{2} - c_t \right\rceil$ and $\left\lfloor \frac{\nonlazy t}{2} + c_t \right\rfloor$,
	\begin{equation}
	\begin{split}
		\dP (N_t = 2n) &= \frac{t!}{(2n)! (t-2n)!} \nonlazy^{2n}(\lazy)^{t-2n}
			\\
            &= \frac{2n +1}{t-2n}\times\frac{\lazy}{\nonlazy} \dP(N_t = 2n+1) \\
			&\leq \frac{ \nonlazy t + 2 c_t+1}{t- \nonlazy t - 2c_t -1} \times \frac{\lazy}{\nonlazy} \dP(N_t = 2n+1) \\
			&\leq \left( 1 + C\sqrt{\frac{\log t}{t}} \right) \dP(N_t = 2n+1),
	\end{split}
	\end{equation}
	where the constant $C$ only depends on $\alpha$, and where we used $c_t/t=\sqrt{(\log t)/t}$.
	Hence, by a computation similar to before, 
	\begin{equation}
	\begin{split}
		q_t (x, z) 
			\leq &\sum_{n = \left\lceil \frac{\nonlazy t}{2} - c_t \right\rceil}^{\left\lfloor \frac{\nonlazy t}{2} + c_t \right\rfloor}  \dP(N_t = 2n) 
				\widebar{q_{2n}} (x, z) \exp \left( C \frac{l c_t}{n} \right)
				+ \dP \left( \left| N_t - \nonlazy t \right| > 2 c_t \right) \\
			\leq &\sum_{n = \left\lceil \frac{\nonlazy t}{2} - c_t \right\rceil}^{\left\lfloor \frac{\nonlazy t}{2} + c_t \right\rfloor}  \dP(N_t = 2n+1) \left( 1 + C\sqrt{\frac{\log t}{t}} \right) q_{2n+1}(y,z) \left(1 + C\frac{\ell c_t}{t} \right)
            \\
			&	+ 2 \exp \left( -8 \log t \right) \\
			\leq &\left( 1 + C\ell\sqrt{\frac{\log t}{t}} \right) q_t(y,z)
				+ 2 \exp \left( -8 \log t \right).
	\end{split}
	\end{equation} 

\end{proof}

The following is an adapted version of \cite[Proposition B.3]{CKKR24}. It `sandwiches' the evolution on an environment with a deterministic initial condition of empirical density $\rho$ between two randomized environments of law $\mu_{\rho\pm\varepsilon}$. This will be the main building block for coupling pairs environments with different initial conditions.

\begin{lemma} \label{Lemma: Soft Local Time Coupling} 
    There exist positive and finite constants $\Cr{lower bound for t in sltc}$, $\Cr{upper bound on t with epsilon in sltc}$, and $\Cr{prob bound for G in sltc}$ such that, for $\rho \in (0,K)$, $\varepsilon \in (0, (K-\rho) \wedge \rho \wedge 1$, and $H, \ell, t \in \dN$ such that $\Cr{lower bound for t in sltc} \ell^2 < t < H/2$ and $(\rho + \varepsilon) \left(\frac{\log t}{t}\right)^{1/2} \ell < \Cr{upper bound on t with epsilon in sltc}\varepsilon$,
    there exists a coupling $\dQ$ of $( \eta^{\rho - \varepsilon}, \eta_t, \eta^{\rho + \varepsilon})$, with $\eta^{\rho \pm \varepsilon} \sim \mu_{\rho \pm \varepsilon}$ and $\eta_t$ sampled under $\Penv^{\eta_0}$, such that if $\eta_0 \in \Sigma$ is such that for any interval $I \subseteq [0,H]$ with $|I| = l$,
    \begin{equation}
        (\rho - \varepsilon/2)l \leq \eta_0(I) \quad 
        (\textit{resp. }  \eta_0(I) \leq (\rho + \varepsilon/2)l
    \end{equation}
    then with $G = \{ \eta^{\rho - \varepsilon} \vert_{[t, H-t]} \preccurlyeq \eta_t \vert_{[t, H-t]}\}$ (resp. $\eta_t \vert_{[t, H-t]} \preccurlyeq \eta^{\rho - \varepsilon} \vert_{[t, H-t]} \}$),
    \begin{equation}
        \dQ(G) \geq 1 - H \exp\left( -\Cr{prob bound for G in sltc} (\rho + \varepsilon)^{-1} \varepsilon^2 \sqrt{t}\right).
    \end{equation}
\end{lemma}
\begin{proof}
    The proof of Lemma \ref{Lemma: Soft Local Time Coupling} is the same as the proof of \cite[Proposition B.3]{CKKR24}, replacing the symmetric heat kernel with the asymmetric heat kernel $\widebar{q_n}$ used in Lemma \ref{Lemma: Heat Kernel}. In particular,~\eqref{eq: soft local time bounds} and~\eqref{eq: max of heat kernel} play the role of (B.9) and (B.13) in \cite{CKKR24}, respectively. 
\end{proof}

\subsection{Proof of the conditions}
Let $J$ be a bounded open interval of $\mathbb{R}_+$ as in Section \ref{subsec:DefRWRDREgeneral}, and consider the APCRW for $\rho \in J$, and fixed values of the other parameters $\alpha, q\in (0,1)$. Recall that $\nu=1$ in this setting.  
\begin{lemma} \label{lemma: density change for APCRW}
Condition \ref{pe:densitychange} holds for the APCRW. 
\end{lemma}
\begin{proof}
    The proof is the same as the proof of Lemma B.5 in \cite{CKKR24}, using Lemma \ref{Lemma: Soft Local Time Coupling} in lieu of Proposition B.3 in \cite{CKKR24} in order to translate to our environment. This allows us to couple the empirical distribution of the environment with the invariant Poisson product measures $\mu_{\rho\pm \varepsilon}$, the error term being similar in Lemma \ref{Lemma: Soft Local Time Coupling} and \cite[Proposition B.3]{CKKR24}. Since the invariant measures are the same in both settings, the second error term arising from the fluctuations of their empirical densities coincide. 
\end{proof}

\begin{lemma} \label{lemma: drift for APCRW}
Condition \ref{pe:drift} holds for the APCRW.
\end{lemma}
\begin{proof}
    By assumption, $\eta_0\vert_{[-H,H]} \succcurlyeq \eta_0' \vert_{[-H, H]}$. 
    We can couple $\eta$ and $\eta'$ by injectively pairing particles in $\eta'$ with particles at the same site at time 0 in $\eta$ (i.e. particles in $\eta'_0(x)$ can be coupled together with particles in $\eta_0(x)$). 
    Those paired particles will then follow the same asymmetric lazy random walk trajectory, independently of other pairs. 
    Then, let any unpaired particles in $\eta$ follow independent asymmetric lazy random walks. 
    Since each particle in the environment can cover at most one step at a time, 
    no particle from outside $[-H, H]$ can drift into $[-H + t, H - t]$ by time $t$. 
    Thus, the event in \eqref{eq:SEPdriftdeviations} occurs deterministically. 
\end{proof}

\begin{lemma} \label{lemma: couplings for APCRW}
Let $K> 1$ such that $J \subseteq (K^{-1}, K)$, $\rho \in (K^{-1}, K)$, $\varepsilon \in (0, (K-\rho) \wedge \rho\wedge 1$, and $H, \ell, t \in \dN$ be such that $\Cr{lower bound for t in sltc}\ell^2 < t < H/2$ and $(\rho + 3 \varepsilon/2)(t^{-1}\log t)^{1/2}\ell < \Cr{upper bound on t with epsilon in sltc}\varepsilon/4$.
Let $\eta_0, \eta'_0 \in \Sigma$ be such that for every interval $I \subseteq [0,H]$ of length $\ell$, we have that $\eta_0(I) \geq (\rho + 3\varepsilon/4)\ell$ and $\eta'_0(I) \leq (\rho + \varepsilon/4)\ell$. 
We then have a coupling $\dQ$ of $\eta$ and $\eta'$ such that
\begin{equation}
    \dQ \left( \eta'_t|_{[t, H-t]} \preccurlyeq \eta_t|_{[t, H-t]} \right)
    \geq 1 - 4H \exp \left( -\Cr{prob bound for G in sltc} (\rho + \varepsilon)^{-1} \varepsilon^2 \sqrt{t}/4 \right).
\end{equation}

We then also have that condition \ref{pe:couplings} holds for the APCRW when $\nu = 1$. 
Additionally, $\dQ$ is local, meaning that $(\eta_t, \eta'_t)|_{[t, H-t]}$ depends only on the initial conditions $(\eta_t, \eta'_t)|_{[0, H]}$.
\end{lemma}
\begin{proof}
    The proof is the same as the proof of \cite[Lemma B.7]{CKKR24}, using  Lemma \ref{Lemma: Soft Local Time Coupling} instead of Proposition B.3 (in \cite{CKKR24}) in order to translate to the asymmetric case, and noting that P1 in \cite{CKKR24} is the same as \ref{pe:markov}.

\end{proof}

\begin{lemma} \label{lemma: c2}
Condition \ref{pe:compatible} holds for the APCRW. 
\end{lemma}
\begin{proof}
    The proof is the same as the proof of Lemma B.8 in \cite{CKKR24}, using Lemma~\ref{lemma: density change for APCRW} instead of \cite[Lemma B.5]{CKKR24}, Lemma~\ref{lemma: drift for APCRW} instead of \cite[Lemma B.6]{CKKR24}, Lemma~\ref{lemma: couplings for APCRW} instead of \cite[Lemma B.7]{CKKR24} and having all particles follow asymmetric lazy random walks instead of symmetric lazy random walks.
\end{proof}

\begin{lemma}
Condition \ref{pe:nacelle} holds for the APCRW.
\end{lemma}
\begin{proof}
    
Let $I = [0, \ell]$ be the interval where we have (at time $0$) the extra particle given by the conditions of \ref{pe:nacelle}. Denote by $Z_t$ the position of this particle at time $t$.

Let $\Qcoupling$ be a coupling between $\eta$ and $\eta'$ on $I$. Under $\Qcoupling$, pair particles in $\eta'_0(I)$ injectively with particles in $\eta_0(I)$ in the same location, and have them undergo the same random walk (independently of other particles). Then let any extra particles in $\eta(I)$ perform random walks independent of everything else. There will not be any extra particles in $\eta'(I)$. 

First, under $\dQ$, since each particle can only take one step at a time, taking $k \geq 1$, no particle from outside the interval $[-H, H]$ at time $0$ can enter the interval $[-H + \ell, H -\ell] \subseteq [-H + 2 k \ell , H - 2 k \ell]$. 
So $\eta_s |_{[-H + 2 k \ell , H - 2 k \ell]} \succcurlyeq \eta'_s |_{[-H + 2 k \ell , H - 2 k \ell]}$ for all $s \in [0,l]$ almost surely. 
Thus, the final requirement is satisfied, and all that remains in showing that $\Qcoupling \left(\eta_\ell(x) > 0, \eta'_\ell(x) = 0 \right) \geq 2\delta$. 

There are at most $6(\rho + 1)\ell$ particles in the interval $[-3\ell + 1, 3\ell]$ at time $0$, so at time $\ell-1$ there are not more than $6(\rho + 1)l$ particles in $\{-1,0,1,2\}$. Then we require all of these particles to take a step so that $\eta_\ell(x) = 0$. This occurs with probability at least $\left( \alpha \min(q, 1-q) \right)^{6(\rho+1)l}$, by specifying the step of each particle. 

Then finally, whether $x=0$ or $x=1$, note that there is a path $\left(\mathbf{z}_s\right)_{s=0}^\ell$ for the extra particle  starting at $Z_0 \in I$ such that $Z'_\ell = x$. So $\dQ(Z_s = \mathbf{z}_s :0 \leq s \leq \ell) \geq r^{\ell}$ with $r:=\min(\alpha q, \alpha (1-q), 1-\alpha)$.

Putting together these restrictions we end up with 
\begin{equation}
\begin{split}
    \dQ\left( \eta_\ell(x) >0, \eta'_\ell(x) = 0 \right)
    &\geq \dQ \left(\eta'_\ell(x) = 0,  Z_s =\mathbf{z}_s: 0 \leq s \leq \ell \right) \\
    &\geq r^{6(\rho + 1) \ell} 
        r^\ell \\
    &\geq \left( r^2\right)^{6(\rho + 1)\ell}.
\end{split}
\end{equation}
\end{proof}


\end{document}